\documentclass[11pt,leqno]{amsart}
\usepackage{amsthm,amsfonts,amssymb,amsmath,oldgerm}
\numberwithin{equation}{section}
\usepackage{fullpage}
\usepackage{color}
\usepackage{calrsfs}
\DeclareMathAlphabet{\pazocal}{OMS}{zplm}{m}{n}
\usepackage{graphicx, xcolor}
\usepackage{subcaption}

\def\eps{\varepsilon }

\newcommand\R{\mathbb R}

\def\eps{\varepsilon}

\newcommand\br{\begin{remark}}
\newcommand\er{\end{remark}}
\newcommand\bp{\begin{pmatrix}}
\newcommand\ep{\end{pmatrix}}
\newcommand{\be}{\begin{equation}}
\newcommand{\ee}{\end{equation}}
\newcommand\ba{\begin{equation}\begin{aligned}}
\newcommand\ea{\end{aligned}\end{equation}}

\newcommand{\bap}{\begin{app}}
\newcommand{\eap}{\end{app}}
\newcommand{\begs}{\begin{exams}}
\newcommand{\eegs}{\end{exams}}
\newcommand{\beg}{\begin{example}}
\newcommand{\eeg}{\end{exaplem}}
\newcommand{\bpr}{\begin{proposition}}
\newcommand{\epr}{\end{proposition}}
\newcommand{\bt}{\begin{theorem}}
\newcommand{\et}{\end{theorem}}
\newcommand{\bc}{\begin{corollary}}
\newcommand{\ec}{\end{corollary}}
\newcommand{\bl}{\begin{lemma}}
\newcommand{\el}{\end{lemma}}
\newcommand{\bd}{\begin{definition}}
\newcommand{\ed}{\end{definition}}
\newcommand{\brs}{\begin{remarks}}
\newcommand{\ers}{\end{remarks}}

\newcommand{\ZZ}{{\mathbb Z}}

\newtheorem{theorem}{Theorem}[section]
\newtheorem{proposition}[theorem]{Proposition}
\newtheorem{corollary}[theorem]{Corollary}
\newtheorem{lemma}[theorem]{Lemma}

\theoremstyle{remark}
\newtheorem{remark}[theorem]{Remark}
\theoremstyle{definition}
\newtheorem{definition}[theorem]{Definition}

\newtheorem{example}[theorem]{Example}

\newcommand{\beq}{\begin{equation}}
\newcommand{\eeq}{\end{equation}}
\title{
Game-theoretic analysis of Guts Poker
}

\author{Luca Castronova}
\address{Bloomington High School South, Bloomington, IN 47405}
\email{luca.castronova@gmail.com}
\thanks{Research of L.C. was partially supported under Indiana Universities Summer 2021 REU program, 
under NSF grant no. DMS-1757857.}
\author{Yijia Chen}
\address{Indiana University, Bloomington, IN 47405}
\email{chenyiji@iu.edu}
\thanks{Research of Y.C. was partially supported by the Indiana University Mathematics Department Research Travel
and Support Fund}
\author{Kevin Zumbrun}
\address{Indiana University, Bloomington, IN 47405}
\email{kzumbrun@indiana.edu}
\thanks{Research of K.Z. was partially supported
under NSF grant no. DMS-0300487}

\begin{document}

\begin{abstract}
We carry out a game-theoretic analysis of the generalized recursive game ``Guts,'' a variant of poker featuring 
repeated play with possibly growing stakes.  An interesting aspect of such games is the need to account for 
funds lost to all players if expected stakes do not go to zero with the number of rounds of play. 
We provide a sharp, easily applied criterion eliminating this scenario, under which one may compute a value for 
general games of this type.
Using this criterion, we determine an optimal ``pure'' strategy for a 2-player continuous version of guts, 
consisting of a simple threshold criterion.  For the $n$-player continuous version, $n\geq 3$, 
we determine an optimal threshold strategy against ``bloc play'' in which players 
2-$n$ pursue identical strategies, giving nonnegative return for player 1.
Against general collaborative strategies of players 2-$n$, we show that player 1 cannot force a nonnegative return.
It follows that there exists a nonstrict symmetric Nash equilbrium, but this equilibrium is not strong.

Finally, we obtain an analogous partial result for the original discrete 2-player game, 
determining an optimal, pure, strategy under the restriction that pure strategies be of threshold type.
\end{abstract}

\date{\today}
\maketitle
\tableofcontents

\section{Introduction}\label{s:intro}
Recursive games, first studied by Everett \cite{E}, consist of ``game elements'' analogous to Markov states.
Starting in an initial state, players choose a first-round strategy resulting in a random outcome, which consists of
either termination of the game with a payoff associated to that state, or else redirection to another state/game element, after which play is reinitiated in that element.  Thus, the game may have many repeated stages, in principle infinitely many, and the stakes of the game vary with the current state.
This is closely related to the notion of stochastic game introduced earlier by Shapley \cite{Sh1}, 
which likewise features game elements, but in which play at a certain element results in transition to a 
new state {\it and} a payoff associated to the current state, but without termination.  A stochastic
game is either played for a fixed number of rounds, or continued indefinitely, with 
total payoff defined as lim inf of stage averages, or a discounted sum of stage payoffs.
The total payoff of a recursive game is defined simply as the (undiscounted) sum of stage payoffs, 
or their lim inf: for a finite-state recursive game with finite state payoffs, clearly finite.

Note that in terms of expected payoffs at each stage, a recursive game is equivalent to a stochastic game
with varying stakes, in which the payoff for state $i$ is the probability $0\leq \theta_i\leq 1$ of termination
times the expected payoff upon termination, the transition probability $\Pi_{ij}$, $\sum_j \Pi_{ij}=1$
from state $i$ to state $j$ is given by $\pi_{ij}/(1-\theta_i)$, where $\pi_{ij}$ is the 
transition probability for the original recursive game,
satisfying $\sum_j \pi_{ij}=1-\theta_i$, and the stakes are multiplied by factor $0\leq (1-\theta_i)\leq 1$ 
for the next stage; that is, future payoffs are adjusted by factor $(1-\theta_i)$.
By this means,
a standard recursive game may be recast as a stochastic game with variable but nonincreasing stakes.

Here, motivated by the example of Guts Poker described below,
we define a {\it generalized recursive game} to be a stochastic game with variable and possibly increasing
stakes, and study this new class of games.
We expect that other interesting applications may be found in business and economics.

Existence or nonexistence of a minimax value for such a game is an interesting question, as is the determination of 
the value and optimal strategies.  A natural approach pioneered by Everett is to consider the question as a fixed-point iteration from values of game elements to themselves, with the mapping determined by von Neumann's minimax principle for ordinary games, viewing the recursive game as an ordinary one with payouts for different game elements given by the 
prescribed input values of each element.    
Should a unique fixed point exist, one expects that the game element values fixed by this process
represent the values of each element in the sense of the original recursive game.
A further step is to show that this is indeed true under reasonable definitions of value.

In this paper, we study a simple generalized recursive game consisting of a single game element, namely, the poker 
variant known as ``Guts'' \cite{W1}, which can be played in $2$- and multi-player versions.

This has interest both theoretical and practical.
On the theoretical side, this is a case for which expected stakes may be nondecreasing as the game goes on,
hence the ``value map'' of Everett is noncontractive and fixed points in general nonunique.
Moreover, nondecreasing stakes leave open the possibility that some portion of funds may be effectively lost to all players, as inaccessible antes in a nondecreasing pot, making the game in principle non-zero sum.
We sidestep these issues by the introduction of a simple and sharp direct criterion determining whether a player can force a nonnegative return, the main issue of interest in many cases; see Theorem \ref{tthm}, Section
\ref{s:unbounded}.
This is phrased in the general framework of generalized recursive games with a ``buyout option'' by which a player may force termination of the game for a fee given by a fixed proportion of the current pot.
This could be zero for example, corresponding to the case that a player is allowed to withdraw their ante and leave the game,  or one in the case (considered here) that a player may leave the game but forfeits their ante.
Other possibilities may be imagined: for example a ``tax'' applied by house or casino.

On the practical side, this is a popular poker game played by many, hence a winning strategy is of inherent interest in applications.
To this end, we first approximate the game by a simplified {\it continuous model}, replacing discrete probabilities of different hands by continuous uniformly distributed probabilities on $[0,1]$.
After analyzing the continuous model, we return to the discrete analysis viewed as a small perturbation.
For the $2$-player, or ``heads-up'' game, we obtain a complete solution for the continuous 
model, consisting of a ``pure'' or deterministic ``go/no-go'' strategy in which the players chooses between their two options in the game (holding or dropping, described just below) according as the value of their hands is above or 
below a certain threshold; see Section \ref{s:analysis_2}.
For the 2-player game this threshold is the ``median'' hand for which a player has equal probability receiving a hand greater or less than the optimal one.
For the corresponding discrete game, restricting to threshold-type pure strategies, we obtain a similar threshold
solution, in which the threshold value is roughly but not exactly that of the median hand, having a slight
shift due to the fact that cards are drawn without replacement, changing conditional probabilities.
We believe but have not shown that this solution is optimal also against non-threshold strategies of the opponent, 
which for the most part are dominated by threshold strategies; see discussion, Section \ref{s:disc}.

For the n-player game, $n\geq 3$, we treat the continuous model only. For 3 players, we compute the full payoff function 
and use this to show that (i) there is a go/no-go strategy generalizing that of the 2-player game guaranteeing a nonnegative return against ``bloc'' strategies, in which players 2 and 3 choose identical strategies in each round, and
(ii) against general player 2-3 strategies, player 1 cannot force a nonnegative return by any strategy whatsoever (either pure or ``mixed'', i.e., random); see Section \ref{s:analysis_3}.
That is, considered as a 2-player game of player 1 against a coalition of players 2 and 3, the game 
has a strictly negative return for player 1.
Though we do not carry it out here, we expect that this implies by continuity of return with respect to payoff function 
that the discrete game also has a strictly negative return.
For the $n$-player continuous game, we compute a restricted payoff function against bloc strategies and small perturbations thereof, which turns out to be enough to make the same conclusions (i)--(ii) as in the 3-player case;
see Section \ref{s:n}.

We note that Guts is a {\it symmetric game}, i.e., identical from the perspective of each different player, hence
the best possible outcome forceable by player 1 (or any player) is a nonnegative return.
Expressed in the language of Nash' theory of $n$-player noncooperative games \cite{N,W2}, 
our results show that for any $n$, continuous $n$-player Guts possessess a symmetric {\it Nash equilibrium}
with expected value zero for all players, consisting of identical ``pure'' threshold strategies for
which we provide a simple explicit formula for the threshold value as a function of $n$.
For finite matrix games, {\it any} symmetric game possesses a symmetric equilibrium \cite{N}, which in
the zero-sum case returns zero to each player; however, in the present more general case this is far from obvious.
Moreover, the symmetric Nash equilibrium guaranteed in the finite case is in general of ``mixed'', or random,
rather than pure type; this special feature is another interesting aspect of Guts. 

Recall that a Nash equilibrium is a collection of (in general mixed) strategies from which departure by a single player with all other players holding fixed will result in an equal or lesser return for the deviating player;
thus, in principle, players have a motivation to remain at this equilibrium point.
However, this ignores the possibility of players working in concert to improve their joint outcome.
A {\it strong} Nash equilibrium- which may or may not exist- 
is a stronger notion taking into account such possible coalitions, requiring that no {\it subset} of players,
or coalition, may jointly profit by deviation from the equilibrium.
The two concepts coincide in the case $n=2$. For $n\geq 3$, our results show that the symmetric Nash equilibrium
{\it is not a strong equilibrium}, but rather
can be destroyed by the action of a coalition: namely, the coalition consisting of
players 2-$n$.

\medskip

{\bf Acknowledgement:}
We thank the Indiana University Mathematics Department, and particularly chair K. Pilgrim and REU coordinator
D. Thurston, together with the National Science Foundation, for their generous support of this undergraduate research
project during difficult times.
The open source Python environment was invaluable in numerical experiments supporting our analysis.
The graph in figure \ref{negativefig} was made using the Desmos graphing calculator package.
Thanks to M. Lewicka for suggesting the treatment of the ``Weenie rule'' 
carried in Appendix \ref{s:weenie}.
Finally, thanks to Jacob Platnick for suggesting the 
erminology ``generalized recursive game,''
and to Jacob and Jay Lee for discussions in the course of the followup study \cite{BLPWZ} that
influenced this revision.

\section{Description of the game}\label{s:desc}
In the card game of ``Guts,'' players are dealt at random a $2$-, $3$-, or $5$-card hand, depending on variants,
with hands strictly ordered in value. All players ante a common amount to a central pot.
After viewing the cards, players hold their hands face down above the table, 
and, upon a count or signal, simultaneously ``drop'' or ``hold'' their cards.
The players who have dropped are out of play.  If only a single player holds, that player wins the pot and play is
over for that round.  If $m>1$ players hold, the one with highest value hand wins the pot, and the others must
``match'' the pot, or forfeit to the pot an amount equal to its former value, the pot thus increasing by factor
$m-1$. The game is then replayed with all players, including those who dropped, for the new pot (with no further ante).
If all players drop, the game is replayed for the original pot, with no further ante.
Play continues until only a single player holds, ending the round.
The often rapidly growing pot and simplicity of play make this an exciting and attractive game for poker play;
the same features make it appealing for game-theoretic analysis.

\subsection{2-player Guts}\label{s:2intro} In 2-player guts, the pot is only replaced, and never increases, since the
number of holding players is $m\leq 2$, hence $m-1\leq 1$.
This simplifies the mathematical situation somewhat, as does the observation that the game will eventually terminate,
with probability one, unless each player follows the strategy to hold on every play.
This is clearly a suboptimal strategy for each individual player, so can essentially be ignored- in particular, as part of an optimal strategy, a player can make sure that this situation does not occur.
Indeed, as we will show, there is a ``pure,'' or deterministic optimal strategy consisting of a fixed card value 
above which the player always holds, and below which they always drop-
roughly the ``median'' strategy in which the determining card value is the one above or below 
which hands occur with probability $1/2$- guaranteeing a nonnegative expected return.
That is, there is a (deterministic!) von Neumann
equilibrium in which each player in this symmetric game can guarantee the same (fair) return $0$.

\subsection{n-player Guts}\label{s:nintro}
For $n$-player guts, $n\geq 3$, the situation is more complicated.  
First of all, the notion of value becomes much trickier for any game with $n\geq 3$ \cite{N}.
We will investigate what is the expected return that player 1 can guarantee against any fixed collection of strategies for players 2-$n$, which may be contingent on outcomes of earlier rounds, but must be chosen before play begins.

This is equivalent to permitting collaboration between opponents 2-$n$, but without communication of card values.
A second issue is that the pot can now in principle grow without bound, making a mathematical theory of 
value more tricky.
In particular, a game could continue indefinitely with nonvanishing amount of funds remaining in the pot undistributed to
any player.
We finesse this issue using the convention that player 1 may at any time ``opt out'' or walk away from the game, forfeiting their chance at winning the pot, equivalent to dropping for every future hand.
This shifts the difficulty to one of determining value for games with a termination fee, or buyout clause.

\section{Modeling and preliminary simplification}\label{s:modeling}
To simplify the initial discussion, we replace discrete card hands by continuous random variables 
$$
p_i, \quad i=1,\dots, n,
$$
uniformly and independently distributed in $[0,1]$, with ordering of continuous ``hands'' given by the standard ordering of the reals.
We shall return to the discrete case later, at least in the case $n=2$.
A strategy for the $i$th player then consists of a measurable subset $S_i \subset[0,1]$ for which the player holds 
if $p_i\in S$ and drops otherwise.

We can simplify this still further by observing that only subsets of form $S=[ p^*,1]$ need be considered, as
they majorize (i.e., give equal or better outcome than) any strategy $S$ with the same measure $|S|=|1-p^*|$. 
This may be seen by the fact that the map $p(s)$ 
defined by $|S \cap [s,1]|=|[p(s),1]|$ 
takes $s\in S$ to $[1-|S|,1]$ and is monotone and measure-preserving, with $p(s)\geq s$,
whence, comparing $p(s)$ to $s$, we find that the outcomes for holding with strategy $[p^*,1]$ are better than
those for $S$, since they have the same probabilities and higher card values.
On the other hand, the outcomes for dropping are the same, since they have the same total probability and value of cards is irrelevant.
{\it Thus, we need only consider single-stage strategies of form $[0,p^*]$.}  Henceforth, we shall drop the set notation altogether, and simply refer to a strategy by its cutoff value $p^*\in [0,1]$, 
with the understanding that a player holds for $p\geq p^*$ and drops for $p<p^*$.

\subsection{Reduction to noncontingent strategy sequences}\label{s:strat1}
We next consider total strategies, consisting of sequences of single-stage strategies for each stage $n$, 
possibly contingent upon the number $n$ of the stage, and the current winnings and stakes at that stage.
A straightforward but crucial further simplification, valid for any generalized recursive game, is
that these sequences may without loss of generality be taken as {\it noncontingent}, i.e., depending only on $n$.
This is a consequence of the dynamical programming principle of Bellman \cite{Sn}.  For, whenever we have a well-defined optimum sequence starting from stage $m$ with stakes $1$ and winnings $0$, this strategy by an affine rescaling is also optimum for arbitrary stakes and winnings.
And, by the dynamical programming principle, together with independence of random outcomes at different stages,
an optimum contingent strategy starting from $m-1$ with stakes $1$ and
winnings $0$ is given by the optimal strategy for the single-shot game $A_{ij}+  B_{ij}R_m$, where
$R_m$ is the optimal expected return starting at stage $m$. 
Thus, (i) we may restrict to consideration of (uncontingent) {\it strategy sequences} $S_1, S_2, \dots, S_n, \dots$
depending only on the stage $j$ and not on current winnings or stakes, 
where $S_j$ is the single-stage strategy chosen at stage $j$,
and (ii) for such strategy sequences, we need only carry the information of {\it expected} one-shot return and stakes factor, and not individual outcomes in order to compute expected total return.
From here on, we consider only {\it uncontingent} strategy sequences, throwing out the dominated contingent type.

\subsection{One-shot payoff and stakes functions}\label{s:payoff}
%Assume that the original ante is one monetary unit. Viewing the game from the standpoint of player 1, we see
%that the strategy $p^*=0$ guarantees that they will not lose more over the course of the game than their original ante,
%so that their guaranteed return is $\geq -1$.  Define $V$ to be the supremum of the expected returns that can be 
%guaranteed for player 1 by the various possibilities of strategies that are available to them.
%In case the game continues indefinitely without terminating, we will take the viewpoint that player 1 and all other players lose their original antes, since they are permanently ``on hold'' and unavailable.
%As noted in the introduction, this is a situation that is not expected to play a role in the final analysis, and this convention is convenient in assigning a value.
%
%Evidently, with this definition, the value to player 1 of a repeated game is the product of the expected
%multiple of the original pot times $V$.  In the $2$-player game, for which the pot never grows, it is simply $V$.
The above considerations allow us to compute the expected total return $R_1$ to player 1 for a given set of strategies $p_j^*$ for each player $j$ in the first round of play, assuming that strategies for rounds $2,3, \dots$ have 
already been assigned, as 
\be\label{payoff}
%\Psi(p_1^*, \dots, p_n^*)=
R_1=
\alpha (p_1^*, \dots, p_n^*) +
\beta (p_1^*, \dots, p_n^*) R_2,
\ee
where $R_2$ is the expected total payoff starting at stage 1 with stakes 1, $\alpha$ is the expected one-shot return,
and $\beta$ is the expected multiple of the original pot given that play continues, times the
probability of repeated play: in the $2$-player game (for which the pot never grows), 
simply the probability of repeated play.
More generally,
\be\label{gpayoff}
R_{m-1}=
\alpha (p_1^*, \dots, p_n^*) +
\beta (p_1^*, \dots, p_n^*) R_{m},
\ee
where $R_m$ is the return starting at stage $m$ with strategy sequence $S_m, S_{m+1}, \dots$ and initial stakes 1.

%FIX
The first step in our analysis will be to calculate the one-shot payoff and stakes functions $\alpha$ and
$\beta$, which are always computable. Afterward,
we will use \eqref{gpayoff}, together with an appropriate limiting and or truncation scheme as $m\to \infty$
in order to compute the expected payoff for a given strategy sequence, and, eventually,
obtain information on optimal strategy sequences and total return for a given generalized recursive game.
%we will try to use this information to determine the value of $V$, whether by fixed-point theory, direct iteration, or other means, and to deduce an optimal strategy achieving that value.

\subsection{Bookkeeping details}\label{s:detail}
In computing the expected one-shot payoff $\alpha$, or immediate return,
for guts,
we shall take the point of view that the ante is to be paid upon termination of the game.
Thus, for example, if two players hold, then the immediate return to the winning player is the value
$+n$ of the entire pot, and to the losing player $-n$, with the stakes at next round remaining at value $1$,
the multiplier of the pot. The immediate return to any players that drop in this scenario is $0$.

A subtlety of this bookkeeping system occurs when three players hold.
For, then, the pot doubles, effectively paying all players one unit of additional ante in the resulting
higher-stakes game, which they will in fact never have to pay.  So the immediate returns of all players are
incremented by one unit and the stakes- and ante- are changed to 2, exactly balancing out.  Thus, the winning player
receives immediate return $n+1$ and the losing (holding) players receive return $-n+1=-(n-1)$.  Any dropping players
receive $+1$.
This system may seem a bit strange, but, comfortingly, one may check that the immediate return is in every event
zero-sum:  if $r$ players hold, then the stakes are multiplied by $(r-1)$, giving all players an additional
``virtual ante'' of $(r-2)$.  Meanwhile, the single winning player receives $+n$ return while the $(r-1)$ 
losing (holding) players receive $-n$, for a total immediate return of $n(r-2) + n- (r-1)n=0$.
It follows by summation across events that the expected immediate return $\alpha$ is zero-sum as well.

\subsection{Symmetries}\label{s:symmetries}
The terms $\alpha$ and $\beta$ of the one-shot payoff function feature several symmetries, which
can be useful both in checking and in deriving their form.
The stakes multiplier $\beta$, by symmetry of the underlying game, is invariant under
permutations of player strategies, hence a {\it symmetric function} of its arguments.

Likewise, the one-shot payoff function $\alpha$ is, by symmetry of the game, invariant under permutations of 
strategies $p_2^*,\dots, p_n^*$, but {\it not} (since it is computed from the point of view of player 1's profits only)
under permutations involving $p_1^*$.
Combining this with the zero-sum property noted above, we have also
\be\label{nsym2}
\alpha(p_1^*, \dots, p_n^*) + \alpha( p_2^*,p_1^*,\dots, p_n^*) + \dots + \alpha(p_n^*,p_1^*,\dots p_{n-1}^*)=0,
\ee
yielding in particular
\be\label{n0symm}
\alpha(p_1^*,p_1^*, \dots ,p_1^*)=0. 
\ee
and
\ba\label{nsymm_aux}
\alpha (p_1^*,p_2^*, \dots,p_2^*)&= -(n-1) \alpha(p_2^*,p_1^*,p_2^*, \dots, p_2^*)\\
&= 
-(n-1) \alpha(p_2^*,p_2^*, p_1^*,p_2^*,\dots,p_2^*)=\cdots=
-(n-1) \alpha(p_2^*,\dots, p_2^*, p_1^*).
\ea

\section{Analytic framework: value of single-state generalized recursive games}\label{s:value}
We start by a general discussion of the value of 2-player generalized recursive games involving a single state,
hereafter referred to for compactness of writing as a generalized recursive game.
This will suffice also for our treatment of the n-player game, which we have chosen to view
as a 2-player game between player 1 and players 2-$n$.
For simplicity, we carry out the discussion in the setting of finite games, to which the original discrete game
belongs, indicating at the end extensions to the continuous case. 

Recall first the fundamental theorem of zero-sum 2-player finite matrix games due to von Neumann.
A finite 2-player game may be described by its {\it payoff matrix} $A_{ij}$ recording the expected return, or
payoff, to player 1 given that player 1 chooses strategy $i$ and player 2 strategy $j$, where the 
possible strategies (which could consist of a number of complicated steps) of players 1 and 2 are ordered
in a list and indexed by integers $i=1,\dots m$ and $j=1,\dots,n$.
These are known as ``pure'' or ``deterministic'' strategies. Players may also make use of ``mixed'' or ``blended''
strategies \cite{B,vN} in which they choose strategies $i$ with probability $x_i$ and $j$ with probability $y_j$,
$x$ and $y$ independent, leading to a {\it payoff function}
\be\label{stdpayoff}
\alpha(x,y)= \sum_{i,j} x_i A_{ij}y_j
\ee
corresponding to the expected payoff for these choices.
By the zero-sum assumption, the payoff to player 2 is $-\alpha(x,y)$; thus, the goal of player 1 is to maximize
the value of $\alpha$ while the goal of player 2 is to minimize it.
The Fundamental Theorem, stated below, asserts that the maximum payoff that can be forced by player 1 in worst-case scenario is equal to the minimum that can be forced by player 2 in worst-case scenario.
This joint value of maximin and minimax is defined as the {\it value} of the game.  We will denote it
for reference as $Value(A)$.

\begin{proposition}[Fundamental Theorem of Games \cite{vN}]\label{fundthm}
	For any payoff matrix $A\in \R^{m\times n}$,
	\ba\label{fundrel}
	\max_{0\leq x_i\leq 1,\; \sum_{i=1}^m x_i=1}& \min_{0\leq y_j\leq 1,\; \sum_{j=1}^n y_j=1}
	\sum_{i,j} x_i A_{ij}y_j
	=\\
	&\quad
	\min_{0\leq y_j\leq 1,\; \sum_{j=1}^m y_j=n}
	\max_{0\leq x_i\leq 1,\; \sum_{i=1}^m x_i=1}
	\sum_{i,j} x_i A_{ij}y_j.
	\ea
\end{proposition}

The Fundamental Theorem is a corollary of the following functional-analytic result, 
following by the observation that the payoff function $\alpha(x,y)$ is linear in $x$ and $y$.

\begin{proposition}[Minimax Theorem \cite{O}]\label{mmthm}
	Let $X\subset \R^m$ and $Y\subset \R^n$ be compact convex sets, and let $f(x,y): (X,Y)\to \R$ be
	continuous and concave-convex, i.e., concave in $x$ and convex in $y$:
	$f(\cdot, y)$ concave and $f(x, \cdot)$ convex. Then,
	\be\label{minimax}
	\max_{x\in X}\min_{y\in Y} f(x,y)=
	\min_{y\in Y} \max_{x\in X} f(x,y).
	\ee
\end{proposition}

\br\label{lessrmk}
In the proof of Proposition \ref{mmthm} \cite{O}, convexity and concavity may be relaxed to
\ba\label{weakcc}
f\big(\theta x_1+ (1-\theta) x_2,y\big)&\geq \min\{f(x_1,y),f(x_2,y)\},\\
f\big(x, \theta y_1+ (1-\theta) y_2\big)&\leq \max\{f(x,y_1),f(x,y_2)\}\\
\ea
for $0< \theta< 1$, with equality only if $f(x_1,y)=f(x_2,y)$
(resp. $f(x,y_1)=f(x,y_2)$). 
\er

We are now ready to discuss the interesting case of {\it single-state generalized recursive games,} 
in which for certain outcomes
the game is replayed with varying stakes.
A single-state generalized recursive game with finitely many strategies may be characterized by {\it two} 
payoff matrices $A$ and $B$, where $A_{ij}$ represents the expected ``one-shot payoff'', or
immediate return to player 1 for a single round of play, 
given that player 1 chooses strategy $1\leq i\leq m$ 
and player 2 chooses strategy $1\leq j\leq n$, and $B_{ij}$ the expected stakes in the next round, i.e., the
sum over all events of the product of probability of replaying times the stakes of the replayed game.
We make the important assumption
\be\label{bpos}
B_{ij}\geq 0,
\ee
meaning that the game is never replayed for negative stakes.

\subsection{Strategies and expected payoff, convergent case}\label{s:strat}
A strategy for a generalized recursive game consists of a (possibly infinite) sequence of strategies
$S_1, S_2, \dots, S_n, \dots$
for the one-shot game represented by $A_{ij}$, $B_{ij}$, which, by the discussion of Section \ref{s:strat},
may be taken to depend only on the stage $n$ of play.
To distinguish this from the notion of one-shot strategy, we will refer to this as a strategy {\it sequence}.
The total return given a pair of opposing strategy sequences
is the sum over all stages of each one-shot payoff multiplied by
the current stakes factor, should this sum converge, and the expected payoff is the sum of expected values at each
stage, should this converge.
With increasing stakes, of course, it is possible that these sums do {\it not} converge, and we will 
have to define expected payoff in a more complicated way; but let us first discuss the illustrative
convergent case.

\subsection{Value and fixed points, convergent case}\label{s:vfix}
Suppose that the expected payoff converges almost surely for each pair of strategy sequences.
Let $\underline V$ denote the supremum of expected returns that can be forced by player 1 by different 
strategy sequences, and let $Value(M)$ denote the value of a (one-shot) matrix game
with payoff matrix $M$.  Then, evidently 
\be\label{rec}
\underline V= Value( A+ B\underline V),
\ee
that is, $\underline V$ is a fixed point of the map 
\be\label{Tmap}
T:V\to Value(A+BV).
\ee
This is just the dynamic programming principle of Bellman \cite{Sn}.
Likewise, the infimum $\overline V$ of expected returns that can be forced by player 2 is a fixed point of $T$.
Note that both or either could be $\pm \infty$ in general.  When $\underline V= \overline V=V$, we say
that the game {\it has value $V$}, similarly as in the one-shot matrix game case.

\subsection{Games with diminishing returns}\label{s:dim}
A particularly straightforward case is that of games with {\it diminishing returns}, i.e., satisfying
\be\label{dim}
\max_{ij}|B_{ij}|=\beta_0<1.
\ee
Since one-shot payoffs are bounded by construction, and stakes diminish at each stage by factor 
at most $(1-\beta_0)<1$, both total return and expected total payoff are convergent series to which the 
above reasoning applies.

\begin{example}\label{repeat}
	If players are required to repeat the same mixed strategy on each successive round, then,
	defining $\alpha:=\sum_{ij}x_iA_{ij}y_j$ and $\beta:=\sum_{ij}x_iB_{ij}y_j$,
	we find by geometric series that the total payoff to player 1 is
	$$
	\alpha+ \beta(\alpha + \beta(\alpha+\beta(\cdots= \frac{\alpha}{1-\beta}.
	$$
	Thus, the payoff function is $f(x,y)=\frac{x^TAy}{1-x^TBy}$. Observing that
	$\frac{\theta \alpha_1 + (1-\theta)\alpha_2}{1- \theta \beta_1+ (1-\theta )\beta_2}$
	for $0<\theta<1$ lies strictly between
	$\frac{\alpha_1 }{1- \beta_1}$ and $\frac{\alpha_2 }{1- \beta_2}$ unless
	$\frac{\alpha_1 }{1- \beta_1}=\frac{\alpha_2 }{1- \beta_2}$,
	we find that $f$ satisfies \eqref{weakcc}, hence, by Remark \ref{lessrmk}, 
	obeys the Minimax Theorem, guaranteeing a unique value of the game.
\end{example}

More generally, we have the following definitive result.

\begin{proposition}[\cite{E}]\label{dimprop}
	For generalized recursive games satisfying \eqref{bpos} and \eqref{dim}, 
	map $T$ of \eqref{Tmap} is contractive, and the game has value equal to
	its unique (finite) fixed point. Moreover, the value may be approximated by
	iteration, as the limit of $\{V_n\}$ defined by $V_{n+1}=T(V_n)$, $V_0=0$.
\end{proposition}

\begin{proof}
	As $\max |A_{ij}|\leq \alpha_0$, finite, and $|B_{ij}|\leq \beta_0<1$, we find by comparison with geometric
	series that the expected value is bounded in absolute value by $\frac{\alpha}{1-\beta_0}$, 
	hence $\underline V$ and $\overline V$ are both finite fixed points of $T$.
	But, evidently, 
	$$
	|T(A+BV_1)-T(A+BV_2)| \leq \max |B_{ij}| |V_1-V_2|\leq \beta_0|V_1-V_2|,
	$$
	hence $T$ is contractive by \eqref{dim}.  It thus has a unique (finite) fixed point $V$ by the Contraction 
	Mapping Theorem, approximable by interation.
By uniqueness, moreover, $\underline V= \overline V=V$.
\end{proof}

\begin{corollary}\label{dimcor}
	For generalized recursive games satisfying \eqref{bpos} and \eqref{dim}, 
a necessary and sufficient condition that the game have value $V=0$ is
that $Value(A)=0$. In particular, value $V=0$ always for a symmetric generalized recursive game with diminishing returns.
\end{corollary}

\begin{proof}
	$V=0$ is a fixed-point of $T$ if and only if $0=Value(A+ B0)=Value(A)$.
	For a symmetric game, $Value (A)=Value (-A)=-Value(A)$, hence $Value(A)=0$, giving the result.
\end{proof}

In the above discussion, in defining expected return as the sum of an infinite series, we
implicitly used the fact that expected payoffs, by \eqref{dim}, converge uniformly independently
of chosen strategies as the number of rounds goes to infinity.
In fact it is not possible for real-world players to continue a game indefinitely; however, this issue
too can be sidestepped using \eqref{dim} by the obervation that stakes remaining to be played converge uniformly
to zero with the number of rounds, so that under any reasonable model of termination the value is arbitrarily close
to that of the complete series.

\subsection{Unbounded games}\label{s:unbounded}
For {\it unbounded games}, in which stakes may possibly increase without bound, we must be a bit more careful than
we have been about accounting of payoffs at intermediate times, since, different from the diminishing returns case,
the remaining stakes are not necessarily going to zero.
For example, in our accounting of Guts, we have computed payoffs by subtracting off the players ante at the termination of the game. Yet, even at intermediate times, these funds are encumbered, or ``owed'' by the player, and should be counted in negative payoff. From this point of view, if the game continues indefinitely, funds corresponding to antes are
indefinitely tied up and effectively lost to all players, so that $\underline V<\overline V$ and the game does not
have a traditional value.

To account for these considerations, we add for a general generalized recursive game $(A,B)$ a ``termination fee''
to the computed value at the $n$th round of $-t$ times the current stakes, where $t\geq 0$,
allowing a player to stop the game at stage $n$ before it is naturally concluded.
This could correspond, as in
the case of Guts as accounted here, simply to book-keeping/loss of an original ante, or it could arise from a
``buyout fee'' that a player must pay in order to exit the game before it is finished.
We will refer to factor $t$ as the ``termination constant.''

With this modification, we may define a truncated game, in which players are required to stop (either naturally,
or manually by execution of the termination clause) at stage less than or equal to some upper bound $n$.
We define $\underline{V}_n$ to be the maximum payoff to player 1 that is forcable by player 1 
in the $n$-truncated game, and $\overline{V}_n\geq \underline{V}_n$ to be the minimum payoff to player 1 that
is forcable by player 2.
Evidently, $\underline{V}_0=-t$, $\overline{V}_0=+t$, while, by the dynamic programming principle,
\be\label{truncrec}
\underline{V}_{n+1}=T(\underline{V}_n), \qquad \overline{V}_{n+1}=T(\overline{V}_n), 
\ee
where $T$ is the value map given in \eqref{Tmap}. This gives well-defined, and in principle computable
lower and upper values $\underline{V}_n$ and $\overline{V}_n$ for each $n$-truncated game.

Taking $n\to \infty$, we then define the total lower and upper values by
\be\label{totalvals}
\underline{V}:=\sup_{n\geq 0} \underline{V}_n, \qquad \overline{V}:=\sup_{n\geq 0} \overline{V}_n.
\ee
Note that this construction, when applied to the motivating example of guts,
amounts to the convention described in the introduction that monetary units in the pot are lost
to all parties unless claimed by finite (natural termination of the game.
It is straightforward, by monotonicity of the value map, to see that either 
$\underline{V}_n\equiv \underline{V}_0=-t$ or else $\underline{V}_n$ is strictly monotone increasing,
with $\underline{V}=\lim_{n\to \infty}\underline{V}_n$ similarly as in the iteration of 
Proposition \ref{dimprop} for the case of diminishing returns.
In particular, in the latter case, $\underline{V}$ is a fixed point of $T$.

\br\label{doublermk}
An example indicating the need for truncation is the (asymmetric)
generalized recursive game consisting of repeatedly flipping 
a coin, with payoff to player 1 of $+1$ for heads and $-1$ for tails and available strategies being
to quit after winning, quit after losing, double the stakes and continue after winning, or double the stakes and
continue after losing.
Without truncation, the strategy of quitting after winning and doubling after losing can be recognized
as the famous ``doubling strategy,'' or martingale betting system, apparently guaranteeing eventual
return of $+1$ to player 1, but in fact not realizable in finite lifespan
(Doob's optional stopping theorem \cite{BW}).
\er

With these definitions, we have the following results, comprising a toolkit for the treatment of
unbounded generalized recursive games.

\begin{theorem}\label{tthm}
	Suppose for an arbitrary single-state generalized recursive game with termination constant $-t<0$, 
	that a certain strategy for player 1 has associated payoffs $\alpha$,  $\beta $ satisfying 
	\be\label{tcond}
	\hbox{\rm $\alpha\geq \alpha_0\geq 0$, $\beta \geq 0$, and
$\alpha \geq t( \beta -1) +\eps$, for some $0<\eps<t$.}  
	\ee
	Then, the expected return $\underline{V}_n$ for the $n$-truncated game satisfies
	\be\label{Vnest}
	\underline{V}_n\geq v_n:=\alpha_0 -  \alpha_0 (1-\eps/t)^{\max\{0,n-1\}} -t(1-\eps/t)^n;
	\ee
	that is, the strategy can force a payoff with lower bound \emph{exponentially converging
	to $\alpha_0$}.

Moreover, the value $\underline V$ forcable by player 1 is the smallest fixed-point of \eqref{Tmap}
	that is greater than or equal to $-t$, which by \eqref{Vnest} is greater than or equal to $ \alpha_0$.
\end{theorem}

\begin{proof}
	Evidently, \eqref{Vnest} holds for $n=0$, since $\underline{V}_0=-t=v_0$ is the cost
	of immediate exit. Suppose that \eqref{Vnest} holds at value $n$.
	Then, $\underline{V}_{n+1} \geq \alpha_1 + \beta_1v_{n}$.  Noting, by \eqref{tcond}, that
	$\beta_1\leq 1+\frac{\alpha_1-\eps}{t}$, we find that
	$$
	\begin{aligned}
		\underline{V}_{n+1} &\geq \alpha_1 -  \Big(1+\frac{\alpha_1-\eps}{t} \Big)t (1-\eps/t)^n \\
		&= \alpha_1 -  (t+\alpha_1-\eps ) (1-\eps/t)^n\\
		&= \alpha_1 \big(1- (1-\eps/t)^n\big) -  t (1-\eps/t)^{n+1} 
		\geq \alpha_0 (1- (1-\eps/t)^n) -  t (1-\eps/t)^{n+1} ,
	\end{aligned}
	$$
	verifying \eqref{Vnest} at value $n+1$.  By induction, the result holds for all $n\geq 0$.

	Observing that the value-map \eqref{Tmap} applied to $-t$ gives $\alpha_1 -t\beta_1$, 
	the first step of the inductive sequence, and further iterations give successive further steps,
	we see by monotone increase of the value map and \eqref{Vnest} that there can be no fixed point
	between $-t$ and $\underline V$, in particular none $\leq \alpha_0$.
\end{proof}

\br\label{zerormk}
This result gives sufficient conditions for 
a winning outcome $V\geq \alpha_0\geq 0$, though the estimate \eqref{Vnest}
is only a lower bound.  A necessary condition is $\alpha \geq t( \beta -1) $, or
$$
\alpha -t\beta  \geq -t,
$$
since otherwise the value at each round is strictly less than $-t\leq 0$.
\er

\br\label{tcondrmk}
Simple sufficient conditions for \eqref{tcond} are $\alpha_j \geq \alpha_0$ and $0\leq \beta_j \leq 1-\sigma$ for
$\sigma>0$. This is sufficient to treat many games with nonincreasing stakes, for example $2$-player Guts.
When $\alpha_j=0$, $0\leq \beta_j\leq 1-\sigma$ is also necessary for \eqref{tcond}.
More generally, conditions \eqref{tcond} are equivalent to
	\be\label{alst_tcond}
	\hbox{\rm $\alpha\geq \alpha_0$ and $0\leq \beta \leq  1+ \alpha/t -\sigma$, for some $\sigma >0$.}
	\ee
\er

The example $\alpha_j\equiv 0$, $\beta_j\equiv 1-\eps$ gives equality in \eqref{Vnest},
showing that this estimate is sharp.
This indicates a peculiarity of generalized recursive games, even those with non-increasing stakes:
that on any finite lifetime, there is a (small) probability that the game will not terminate,
hence there remain funds effectively lost to all players in the unclaimed pot.
Thus, a result like that of \eqref{Vnest} is the best one can expect in a symmetric game; for, if all
players pursue the same strategy as player 1, then $\alpha_j\equiv 0$, while $\beta_j$ in general is nonzero,
hence the outcome for all players is an exponentially diminishing negative payoff of the form of $v_n$
in \eqref{Vnest}.
Note, finally, that the strategy for player 1 may be of ``pure'' type as studied here, or of the general
mixed, random type introduced by von Neumann for general 2-player games; the result does not distinguish 
between these types.

In the special case of a symmetric, or other ``fair'' game, we can say much more.

\begin{corollary}\label{fairthm}
	In the special case that $V=0$ is a fixed-point of the value map \eqref{Tmap}, i.e., $\alpha, \beta\geq 0$,
	necessary and sufficient conditions that $\underline V=0$ are that $\alpha\geq 0$ and
$\alpha \geq t( \beta -1) +\eps$, for some $0<\eps<t$, for the optimal strategy of player 1.
Necessary and sufficient conditions that the game have value $V=\underline V=\overline V=0$
are that the analogous conditions $\alpha\leq 0$ and $-\alpha \geq t( \beta -1) +\eps$ hold for the optimal
	strategy of player 2, in which
case on the saddlepoint of simultaneous optimal solutions $\alpha=0$ and,
if $t\neq 0$, also $\beta\leq 1-\eps<1$.
\end{corollary}

\begin{proof}
	Sufficiency follows from Theorem \ref{tthm}. Necessity follows since
	$\inf(\alpha -  \beta t) = -t $, gives expected return $-t$ at every step, hence $\underline V=-t$.
	A symmetric argument for player 2 gives the corresponding condition
	$\inf(-\alpha -  \beta t) \geq -t $ for the optimal strategy of player 2.
	When players 1 and 2 both play optimal strategies, both of these conditions are in effect, hence, 
	adding, we obtain $-2\beta t\geq -2t+ 2\eps$, or $\beta<1-\eps$.
	Meanwhile $\alpha\geq 0$ and $\alpha\leq 0$ gives $\alpha=0$.
\end{proof}

\br\label{nolossrmk}
The conclusion that $\beta<1$ for optimal strategies when $t\neq 0$ corresponds to the intuition that, unless
remaining stakes diminish to zero as the number of rounds goes to infinity, there will always be $2t$ times
the expected remaining stakes $R$ that is lost to both players, hence a gap of $2tR$ between $\underline V_n$ and
$\overline V_n$. Due to this phenomeon
for unbounded generalized recursive games, we see that the fundamental theorem of games 
quite often may not hold.
When it does hold, it reduces on the saddlepoint of simultaneous optimal strategies effectively to a 
game of diminishing returns.
\er

\br\label{exrmks}
For any termination cost $-t$, $t\geq 0$, one may show using monotonicity of the value map \eqref{Tmap}
that $\underline V$ is less than or equal to the value of any fixed point of $T$ that is $\geq -t$,
in particular, less than or equal to any nonnegative fixed point: equivalently,
there exists no fixed point of the value map between $-t$ and $\underline V$.
Specifically, observing that any fixed point between $-t$ and $\underline V$ must lie between values $\underline V_n$ and 
$\underline V_{n+1}=T(\underline V_n)$ in the increasing sequence $\underline V_n\to \underline V$, we may use monotonicity to obtain a contradiction.
Similarly,  for any termination cost $-t$, $t\geq 0$, one may show that $\overline V$ is greater than or equal to
the value of any fixed point of $T$ that is less than or equal to $t$, in particular greater than or equal to 
any nonpositive fixed point.
\er

%NOTES: as mentioned in accompanying talk, CAN phrase this in terms of fixed pts but it is a little awkward,
%not as clear as the above it seems.  (that is, has to be SAME strategyt guaranteeing \alpha \geq \alpha_0\
%and \alpha -\beta t> -t at the same time... see altpf_unfinished file.
%TODO: hmmm, could we improve/simplify somehow? start with Value(\alpha-\beta t)> -t. Then what needed
%about alpha? well, exactly that simultaneous optimal strategy gives \alpha \geq \alpha_0 it seems...
%else all V_n<0 and fixed pt. <0...
%HMMM, wait, that is only if the player insists on using only one strategy all the time... ??? nontrivial
%after all... ok for talk, too complicated for paper...

\subsection{Extension to the continuous case}\label{s:extensions}
Extensions to the continuous case are straightforward. 
Namely, in place of a finite matrix game with payoff $A_{ij}$
for strategies $i,j \in \ZZ^+$, one may consider a payoff function
$A(x,y)$, where $x$ and $y$ lie in compact subsets
$X \subset\R^m$ and $Y\subset \R^n$. Mixed strategies then take the form
$$
\mathcal{A}(P,Q):=\int_{\R^m\times \R^n}A(x,y) dP(x)dQ(y), 
$$
where $P$ and $Q$ are cumulative distribution functions for probability measures supported on $X$ and $Y$,
respectively.
So long as $A(x,y)$ is continuous on $X\times Y$, the payoff $\mathcal{A}(P,Q)$ then has well-defined 
minimax and maximin values, which are equal, and achieved at optima $P_*$ and $Q_*$; see, e.g.,
\cite[Ch. 6]{D}, \cite{F}. 
That is, the minimax theorem and fundamental theorem of games apply also to in this more general case.
Thus, for continuous payoff and stakes functions $A(x,y)$ and $B(x,y)$ we may define a value map
$T(V):=Value(A+BV)$ similar to \eqref{Tmap} in the finite-strategy case, and go on to carry out
all of the analysis of the section above in this larger, infinite-strategy case.

These conclusions apply in particular to our main example of continuous guts poker, since as we shall show just
below the payoff functions are indeed continuous for this game.
Here, $X,Y=[0,1]\subset \R$ for the 2-player case, and $X=[0,1]$, $Y=[0,1]^{n-1}$ for the $n$-player case,
which we have chosen to treat as a 1 vs. $(n-1)$-player game.

%TODO: add comment about domination NOT sufficient for 2-player result!!!! answering referee (see notes
%for exactly the point...)

\section{Analysis of the continuous 2-player game}\label{s:analysis_2}
Having provided the necessary general framework, we are now ready to analyze the specific game of Guts,
starting with the 2-player case.
We first treat the continuous model, denoting by $p_1^*$ and $p_2^*$ the strategies of player 1 and
player 2, respectively, and seeking to determine the payoff function $\Psi(p_1^*,p_2^*)=\alpha + \beta V$ and,
utimately, the value of and optimal strategy for the game.
We then treat the original discrete model
by an adaptation of the arguments of the continuous case.

\subsection{Payoff function}\label{s:2pay}

\begin{proposition}\label{2payprop}
For continuous 2-player guts, the payoff function is $\Psi(p_1^*,p_2^*)=\alpha + \beta V$, where 
	\be\label{2beta}
\beta(p_1^*, p_2^*)= p_1^*p_2^*+ (1-p_1^*)(1-p_2^*)
\ee
and
	\be\label{2alpha}
	\alpha(p_1^*,p_2^*)=
	\begin{cases}
		(1-2p_1^*)(p_1^*-p_2^*) & p_2^*\leq p_1^*,\\
		(1-2p_2^*)(p_1^*-p_2^*) & p_2^*> p_1^*.
	\end{cases}
	\ee
\end{proposition}

\begin{proof}
It is clear that the game terminates unless both players drop, or both players hold, i.e.,
unless $0<p_1<p_1^*$ and $0<p_2<p_2^*$ or $p_1^*\leq p_1\leq 1$ and $p_2^*\leq p_2\leq 1$.
These are disjoint events with probabilities $p_1^*p_2^*$ and $(1-p_1^*)(1-p_2^*)$. 
Thus, the probability of replaying the game is $p_1^*p_2^*+ (1-p_1^*)(1-p_2^*)$, and, since the size of the pot
does not change in the $2$-player game, we have therefore immediately
$\beta(p_1^*, p_2^*)= p_1^*p_2^*+ (1-p_1^*)(1-p_2^*)$.

The determination of $\alpha$ requires consideration of a number of different cases.
	As observed previously, play repeats only if both players hold or both players drop, from which 
	\eqref{2beta} 
	mmediately follows.
	For $\alpha(p_1^*,p_2^*)$ with $p_2^*\leq p_1^*$, there are five cases:

	(i) $p_1\leq p_1^*$ and $p_2> p_2^*$ (drop-hold): player 2 wins, expected return $-1$,
	probability $p_1^*(1-p_2^*)$.

	(ii) $p_1\leq p_1^*$ and $p_2\leq p_2^*$ (drop-drop): both players drop, expected return $0$.

	(iii) $p_1, p_2> p_1^*$ (hold-hold): both hold, fair game, expected return $0$.

	(iv) $p_1> p_1^*$, $p_2\leq p_2^*$ (hold-drop): player 1 wins, expected return $+1$, probability
	$(1-p_1^*)p_2^*$.

	(v) $p_1> p_1^*$, $p_2^*\leq p_2 < p_1^*$ (hold-hold): both players hold, player 1 wins, 
	expected return $+2$, probability $(1-p_1^*)(p_1^*-p_2^*)$.

	Summing products of returns against probabilities, we obtain expected return
	$$
	\alpha(p_1^*,p_2^*)^*=
	p_1^*(1-p_2^*)(-1)+ (1-p_1^*)p_2^*(+1) + (1-p_1^*)(p_1^*-p_2^*)(+2)
	=(1-2p_1^*)(p_1^*-p_2^*)
	$$
	as claimed. To treat the case $p_1^*\leq p_2^*$, we observe by \eqref{nsymm_aux} that
	$ \alpha(p_1^*,p_2^*)^*=-\alpha(p_2^*, p_1^*), $ 
	which in this case gives (from the computation just above)
	$ \alpha(p_1^*,p_2^*)^*= -(1-2p_2^*)(p_2^*-p_1^*)= (1-2p_2^*)(p_1^*-p_2^*)$.
\end{proof}

{\bf Note:}
The payoff functions \eqref{2alpha}, \eqref{2beta} are analogous to payoff matrices $A$, $B$ in the finite
generalized recursive case, Section \ref{s:value}, describing the outcome of two ``pure'', or deterministic strategies 
$p_1^*$ and $p_2^*$. More generally, one may consider ``mixed'', or random, strategies consisting of
probability measures $d\mu_1^*$ and $d\mu_2^*$ on $p_1^*$ and $p_2^*$, 
for which the expected return is given by
\be\label{2mixed}
\int_0^1 \int_0^1 \alpha(p_1^*,p_2^*)d\mu_1^*d \mu_2^*.
\ee

In our analysis here, we shall not require this full generality, but only need to consider pure strategies
or finite random combinations of them: that is, {\it discrete probability theory}.

\subsection{Alternative computation}\label{s:2alt}
We mention also a different way of computing $\alpha$ that reduces the number of cases, based on perturbation
from the symmetric case.
We will make good use of this approach in more complicated situations later on.
Take without loss of generality $p_1^*\leq p_2^*$. By symmetry, $\alpha(p_2^*,p_2^*)=0$.  Thus, we can write
$\alpha(p_1^*,p_2^*)$ as the difference
\be\label{diffeq}
\alpha(p_1^*,p_2^*)- \alpha(p_2^*,p_2^*).
\ee

Note that this difference is zero event-by-event except when $p_1^*\leq p_1\leq p_2^*$, 
since otherwise the behavior of players 1 and 2 is identical for both strategy pairs 
$(p_1^*,p_2^*)$ and $(p_2^*,p_2^*)$. 
	Thus, we may condition on the case
	\be\label{condition}
	p_1^*\leq p_1\leq  p_2^*.
	\ee

	There are two subcases: (i) $p_2\geq p_2^*$, in which case player 2 holds and, because of \eqref{condition},
	always wins. (ii) $p_2<p_2^*$, in which case player 2 drops and thus always loses, {\it independent of
	the value of $p_1$ within range \eqref{condition}}.
	Meanwhile, the difference in payoff \eqref{diffeq} for player 1 between strategy $p_1^*$ and $p_2^*$
	is, by \eqref{condition}, the difference between player 1 holding and dropping: for
	case (i) (since they lose the whole pot if they hold but only their ante if they drop) $(-2)-(-1)=-1$.
	for case (ii) (since they win the ante if they hold and nothing if they drop) $(+1)-(0)=+1$.

	Computing that case (i) has probability $(p_2^*-p_1^*)(1-p_2^*)$ and case (ii) probability
	$(p_2^*-p_1^*)p_2^*$, we thus have an expected difference in return of
	$$
	(p_2^*-p_1^*)(1-p_2^*)(-1) + (p_2^*-p_1^*)p_2^*(+1) = 
	(p_2^*-p_1^*)(2p_2^*-1), 
	$$
	as claimed. The formula in case $p_1^*>p_2^*$ then follows by symmetry.

	\br\label{C1rmk}
	Interestingly, $\alpha$ is $C^1$ in $p_1^*,p_2^*$, matching at the boundary $p_1^*=p_2^*$.
	This property seems not a priori obvious; however, it can readily be seen by a conditional probability
	argument similar to the differencing argument just given. We will make use of this later on in our
	analysis of the n-player game; see, for example, the proof of Lemma \ref{derivlem}.
	\er

\subsection{Best response payoff and optimal strategy}\label{s:2bestresponse}
The (pure) {\it best response payoff} 
$$
R_2(p_1^*):= \min_{p_2^*} \alpha(p_1^*,p_2^*)
$$
is defined as the optimum (i.e., smallest) one-shot payoff $\alpha$ forceable by player 2 against a 
given pure strategy $p_1^*$ chosen by player 1.

\begin{lemma}
The best response payoff for $\alpha$ as given by \eqref{2alpha} is
\be\label{br}
R_2(p_1^*):= \min_{p_2^*} \alpha(p_1^*,p_2^*)=
\begin{cases}
	- \frac{(1-2p_1^*)^2}{8}<0, & p_1< 1/2,\\
	(1-2p_1^*) p_1^* \leq 0 , & p_1\geq 1/2.
\end{cases}
\ee
\end{lemma}

		%(1-2p_1^*)(p_1^*-p_2^*) & p_2^*\leq p_1^*,\\
		%(1-2p_2^*)(p_1^*-p_2^*) & p_2^*> p_1^*.
\begin{proof}
	For $p_2^*\leq p_1^*$, $\alpha$ is linear in $p_2$ with slope $(1-2p_1^*)$, hence $\min_{p_1^*}\alpha$
	is achieved at $p_2^*=0$ or $p_2^*=p_1^*$ according as $p_1^*\geq 1/2$ or $p_1^*\leq 1/2$, and
	thus 
	$$
	\min{0\leq p_2^*\leq p_1^*}\alpha(p_1^*,p_2^*)=\begin{cases} 0 & p_1^*\leq 1/2,\\
		(1-2p_1^*)p_1^* <0 & p_1^* > 1/2.\\
		\end{cases}
	$$

	For $p_2^*\geq p_1^*$, on the other hand,  $\alpha$ is quadratic in $p_2$ and convex, 
	with zeros at $p_2=1/2$ and $p_2=p_1$.
	For $p_1^*\geq 1/2$, therefore, its minimum is achieved at $p_2^*=p_1^*$, with value $0$,
	while for $p_1^*< 1/2$ its minimum is achieved at the interior critical point $(2p_1^*+1)/4$ 
	given by the average of $p_1^*$ and $1/2$, with value $- \frac{(1-2p_1^*)^2}{8}<0$.
	Combining this information, we find for $p_1^*\leq 1/2$, that the minimum of $\alpha(p_1^*,p_2^*)$ with
	respect to $p_2$ occurs at the interior critical point on $(p_1^*,1)$, 
	giving value $- \frac{(1-2p_1^*)^2}{8}<0$, while for $p_1^*>1/2$, it occurs at $p_2^*=0$,
	giving value $(1-2p_1^*)p_1^* <0$.
\end{proof}

\begin{corollary}\label{2opt}
The pure strategy 
\be\label{ppay}
	p_1^*={\rm argmax}_{p_1^*}R_2(p_1^*)=1/2,
\ee
is optimal for player 1, guaranteeing a nonnegative payoff.
\end{corollary}

\begin{proof}
The optimal pure strategy for player 1 is $ p_1^*={\rm argmax}_{p_1^*}R_2(p_1^*)$ 
by definition of the best response payoff, guaranteeing value $max_{p_1^*}R_2(p_1^*)=R_2(p_1^*)$. 
Consulting \eqref{br}, we find that the unique maximum of $R_2$ occurs at $p_1^*=1/2$, with value $0$.
Thus, this choice of pure strategy gives a nonnegative one-shot return $\alpha$, 
	which, by symmetry of the game, is optimal. Moreover, with $p_1^*=1/2$, \eqref{2beta} gives
	$\beta(p_1^*, p_2^*)= (1/2)(p_2^*+ (1-p_2^*))= 1/2<1$, hence, by Theorem \ref{tthm}, 
	guarantees together with $\alpha\geq 0$ a nonnegative return for the choice of strategy $p_1^*=1/2$.
\end{proof}

\br\label{saddlermk}
The existence of an optimal pure strategy for $\alpha$, though not its value, may be deduced from the minimax theorem,
Theorem \ref{mmthm}, observing that \eqref{2alpha}
is concave in $p_1^*$ and convex in $p_2^*$.
\er

\br\label{nomixrmk}
Though we did not state it, the pure strategy $p_1^*=1/2$ is the {\it unique }
optimal strategy for player 1.
Evidently it is the unique optimal pure solution, as $p_1^*=1/2$ is a strict maximum for $R_2(p_1^*)$.
Moreover, any mixed strategy will give inferior return.
For, if it contains any $p_1^*>1/2$ it can be penalized by the choice $p_2^*=1/2$.
If, on the other hand, it contains only $p_1^*\leq 1/2$, and is not equal to $1/2$ with probability one,
then it can be penalized by any $p_1^*$ lying strictly between $1/2$ and $\bar p_1$ 
defined as the mean value of $p_1^*$ under this probability distribution.
For, observing that $\alpha$ in \eqref{2alpha} is concave with respect to $p_1^*$, we have by Jensen's Theorem that
$\alpha(\bar p_1, p_2^*)$ is greater than or equal to the mean of $\alpha(p_1,p_2^*)$, i.e., the payoff for
the mixed strategy against $p_2^*$.
But, consulting \eqref{2alpha}, we find for $\bar p_1<p_2^*<1/2$ that $\alpha(\bar p_1, p_2^*)<0$, hence
the mixed strategy is non-optimal.

This same argument shows for any concave-convex payoff function that mixed strategies for either player are no 
better than the pure strategies given by their means, 
an interesting complement to the minimax theorem, Theorem \ref{mmthm}.
For payoff functions concave in the first argument, it shows that mixed strategies for player 1 are majorized
by the pure strategies given by their means.
\er

\br\label{Notermk}
If player 1 pursues the optimal strategy $p_1^*=1/2$, the payoff function reduces
to
	\be\label{redpay}
	\alpha(1/2,p_2^*)=
	\begin{cases}
		0 & p_2^*\leq 1/2 ,\\
		(1-2p_2^*)(1/2 -p_2^*)>0 & p_2^*> 1/2 .
	\end{cases}
	\ee
	Thus, {\it overcautious play} $p_2^*>1/2$ by player 2 is penalized, but {\it reckless play}
	$p_2^*<1/2$ is not. 
	
	Likewise, $p_2^*=1/2$ penalizes overcautious play $p_1^*>1/2$ by player 1 
	but not reckless play $p_1^*<1/2$;
	this is the reason for the subtlety of the analysis in Remark \ref{nomixrmk}.
	\er
	%HERE

\section{Analysis of the continuous 3-player game}\label{s:analysis_3}
We next consider the $3$-player game.
As described in the introduction, we will view player 1 as competing agains the remaining players $2$-$3$,
who choose a joint strategy without knowledge or communication of each others hands.
Hereafter, we restrict for simplicity to the continuous case.

\subsection{Payoff function}\label{s:3pay}

\begin{proposition}\label{3prop}
For $3$-player guts, the payoff function is $\Psi(p_1^*,p_2^*)=\alpha + \beta V$, where 
	\be\label{3beta}
  \beta=2-p_1^*-p_2^*-p_3^*+2p_1^*p_2^*p_3^*
	\ee
and
	\ba\label{3alpha}
	\alpha(p_1^*,p_2^*,p_3^*= 
	\begin{cases}
		2p_1^*-p_2^*-p_3^*+(p_3^*)^3+3(p_2^*)^2p_3^*-4p_1^*p_2^*p_3^*, &
		p_1^*<p_2^*<p_3^*,\\
		2p_1^*-p_3^*-p_2^*+(p_2^*)^3+3(p_3^*)^2p_2^*-4p_1^*p_2^*p_3^*, &
		p_1^*<p_3^*<p_2^*,\\
		2p_1^*-p_2^*-p_3^*+(p_3^*)^3-3(p_1^*)^2p_3^*+2p_1^*p_2^*p_3^*, &
		p_2^*<p_1^*<p_3^*\\
		2p_1^*-p_2^*-p_3^*+(p_2^*)^3-3(p_1^*)^2p_2^*+2p_1^*p_2^*p_3^*, &
		p_3^*<p_1^*<p_2^*,\\
		2p_1^*-p_2^*-p_3^*-2(p_1^*)^3+2p_1^*p_2^*p_3^*, &
		p_2^*<p_3^*<p_1^*,\\
		2p_1^*-p_2^*-p_3^*-2(p_1^*)^3+2p_1^*p_2^*p_3^*, &
		p_3^*<p_2^*<p_1^*.
	\end{cases}
	\ea
\end{proposition}

\begin{proof}
  As the expected value of the stakes multiplication factor for play in the next round,
	$$
	\begin{aligned}
		\beta&=(p_1^*p_2^*p_3^*)\times 1
		+[p_1^*p_2^*(1-p_3^*)+p_1^*p_3^*(1-p_2^*)+p_2^*p_3^*(1-p_1^*)]\times0\\
		&\quad +[p_1^*(1-p_2^*)(1-p_3^*)+p_2^*(1-p_1^*)(1-p_3^*)+p_3^*(1-p_1^*)(1-p_2^*)]\times1\\
		&\quad + [(1-p_1^*)(1-p_2^*)(1-p_3^*)]\times 2.
	\end{aligned}
		$$
	Simplifying, we obtain \eqref{3beta}.

In computing $\alpha$, there are 6 different situations to consider, which can be reduced to three pairs related by
symmetry in $p_2^*$, $p_3*$.
We list possible returns times their probabilities, then sum, to obtain the expected return for the first example of
each pair, to obtain the $\alpha$-function for that scenario. This yields the second item of the pair by symmetry,
giving 6 different $\alpha$-functions in all.
	
\medskip

	{\it Case 1. ($ p_1^*<p_2^*<p_3^*$ or $ p_1 ^*<p_3^*<p_2^*$).}
	Summing over the table
	$$
	\begin{aligned}
		&(-1)\times p_1^*(p_2^*+p_3^*-2p_2^*p_3^*)\\
		&2\times (1-p_1^*)p_2^*p_3^*\\
		&(-3)\times( (p_2^*-p_1^*)(1-p_2^*)p_3^*+(p_3^*-p_1^*)p_2^*(1-p_3^*))\\
		&(-2)\times(p_3^*-p_1^*)(1-p_2^*)(1-p_3^*)\\
		&1\times (1-p_3^*)^2(p_3^*-p_2^*),
	\end{aligned}
	$$
we obtain

$$
	\begin{aligned}
		\alpha_1(p_1^*<p_2^*<p_3^*)&= 2p_1^*-p_2^*-p_3^*+(p_3^*)^3+3(p_2^*)^2p_3^*-4p_1^*p_2^*p_3^*,\\
		\alpha_2(p_1^*<p_3^*<p_2^*)&= 2p_1^*-p_3^*-p_2^*+(p_2^*)^3+3(p_3^*)^2p_2^*-4p_1^*p_2^*p_3^*.
	\end{aligned}
	$$

	\medskip

	{\it Case 2. ($ p_2^*<p_1^*<p_3^*$ or $ p_3 ^*<p_1^*<p_2^*$).} Similarly, we compute
	$$
	\begin{aligned}
		&3\times (1-p_1)(p_1-p_2)p_3\\
		&2\times (1-p_1)p_2p_3\\
		&(-1)\times (p_1p_2(1-p_3)+p_1p_3(1-p_2))\\
		&(-3)\times (1-p_3)(p_3-p_1)p_2\\
		&(-2)\times (1-p_3)(p_3-p_1)(1-p_2)\\
		& 1\times (1-p_3)^2(p_3-p_2),\\
	\end{aligned}
	$$
	giving
$$
	\begin{aligned}
		\alpha_3(p_2^*<p_1^*<p_3^*)&= 2p_1^*-p_2^*-p_3^*+(p_3^*)^3-3(p_1^*)^2p_3^*+2p_1^*p_2^*p_3^*,\\
		\alpha_4(p_3^*<p_1^*<p_2^*)&= 2p_1^*-p_2^*-p_3^*+(p_2^*)^3-3(p_1^*)^2p_2^*+2p_1^*p_2^*p_3^*.
	\end{aligned}
	$$

	\medskip

	{\it Case 3. ($ p_2^*<p_3^*<p_1^*$ or $ p_3 ^*<p_2^*<p_1^*$).} Finally, we have
$$
	\begin{aligned}
		&(-1)\times (p_1^*p_2^*(1-p_3^*)+p_1^*p_3^*(1-p_2^*))\\
		&2\times (1-p_1^*)p_2^*p_3^*\\
		&3\times ((1-p_1^*)(p_1^*-p_2^*)p_3^*+(1-p_1^*)p_2^*(p_1^*-p_3^*))\\
		&4 \times (1-p_1^*)(p_1^*-p_2^*)(p_1^*-p_3^*)\\
		&1\times (1-p_1^*)^2(2p_1^*-p_2^*-p_3^*),\\
	\end{aligned}
	$$
giving 
$$
	\begin{aligned}
		\alpha_5(p_2^*<p_3^*<p_1^*)&= 2p_1^*-p_2^*-p_3^*-2(p_1^*)^3+2p_1^*p_2^*p_3^*,\\
		\alpha_6(p_3^*<p_2^*<p_1^*)&= 2p_1^*-p_2^*-p_3^*-2(p_1^*)^3+2p_1^*p_2^*p_3^*,
	\end{aligned}
	$$
	completing the proof.
\end{proof}

\br\label{checksumrmk}
For $p^*_1<p^*_2<p^*_3$, we have
$$
\begin{aligned}
	\alpha(p_1^*,p_2^*,p_3^*)&=\alpha_1=2p_1^*-p_2^*-p_3^*+(p_3^*)^3+3(p_2^*)^2p_3^*-4p_1^*p_2^*p_3^*\\
	\alpha(p_2^*,p_1^*,p_3^*)&=\alpha_2=2p_2^*-p_1^*-p_3^*+(p_3^*)^3-3(p_2^*)^2p_3^*+2p_1^*p_2^*p_3^*\\
	\alpha(p_3^*,p_1^*,p_2^*)&=\alpha_3=2p_3^*-p_2^*-p_1^*-2(p_3^*)^3+2p_1^*p_2^*p_3^*,\\
\end{aligned}
$$
  hence, summing, $\alpha(p_1^*,p_2^*,p_3^*)+\alpha(p_2^*,p_1^*,p_3^*)+\alpha(p_3^*,p_1^*,p_2^*)=0$,
  verifying symmetry \eqref{nsym2}.
  \er

  \subsubsection{Bloc case}\label{s:3bloc} 
  Specializing Proposition \ref{3prop} to the case of bloc strategies $p_2^*=p_3^*$ for players 2-3, 
  we obtain the following result generalizing Proposition \ref{2payprop} of the 2-player case.

  \begin{corollary}\label{3bloccor}
	  For bloc strategies,
	  \be\label{3blocbeta}
	  \beta(p_1^*,p_2^*,p_2^*)=2-p_1^*-2p_2^*+2p_1^*(p_2^*)^2
	  \ee
	  and
	  \be\label{3blocalpha}
	  \alpha(p_1^*,p_2^*,p_2^*)= 
	  \begin{cases}
		  2(p_2^*-p_1^*)\big(2(p_2^*)^{2}-1\big),& p_1^*\leq p_2^*\\
  2(p_2^*-p_1^*)((p_1^*)^2+ p_1^*p_2^*-1) &  p_1^*\geq p_2^*.
	  \end{cases}
	  \ee
  \end{corollary}

  \br\label{blocproprmk}
  For the bloc case as in the 2-player case, $\alpha$ is $C^1$, matching at the boundary $p_1^*=p_2^*$,
  and is concave-convex: that is, concave in $p_1^*$ and convex in $p_2^*$.  It follows from the minimax theorem,
  Theorem \ref{mmthm}, that there exist pure strategies $p_1^*$ and $p_2^*$ forcing $\alpha\geq 0$ and
  $\alpha\leq 0$, respectively.
  \er

  \subsection{Best response function and optimal strategy}\label{s:3br}
  With payoff functions in hand, we now investigate optimal responses and strategies for the 3-player game,
  both bloc and otherwise.

  \subsubsection{Bloc case}\label{s:3blocstrat}
  We start by identifying the optimal strategies predicted in Remark \ref{blocproprmk}.

  \begin{proposition}[Optimal bloc strategy]\label{3blocprop}
	  For $p_2^*=p_3^*$, $p_1^*=1/\sqrt{2}$ is optimal, guaranteeing nonnegative return, 
	  with
	  $\alpha\geq 4(p_2^*-1/\sqrt{2})^2(p_2^*+ 1/\sqrt{2})$ for $1/\sqrt{2}\leq p_2^*$ and
	  $\alpha \geq 2(p_2^*-1/\sqrt{2})^2/\sqrt{2}$ for   $1/ \sqrt{2}\geq p_2^*$.
	  Likewise, $p_2^*=1/\sqrt{2}$ is optimal, guaranteeing nonpositive return.
  \end{proposition}

  \begin{proof}
	  Direct substitution of $p_1^*=1/\sqrt{2}$ into \eqref{3blocalpha} yields the stated bounds on $\alpha$
	  giving $\alpha\geq 0$.
	  Likewise, substituting $p_1^*=1/\sqrt{2}$ into \eqref{3blocbeta} and completing the square yields 
	  $$
	 \beta=
	  2- 1/\sqrt{2} -2p_2^*+\sqrt{2}(p_2^*)^2= c+ \sqrt{2}(p_2^*-1/\sqrt{2})^2,
	 $$
	  whence $\beta-\alpha \leq c + (\sqrt{2}-d_\pm) (p_2^*-1/\sqrt{2})^2$
	  with value $c= 2(1-1/\sqrt{2})<1$ and $d_\pm= 4/\sqrt{2}, 2/\sqrt{2}$.
	  In particular, $\sqrt{2}-d_\pm \leq 0$, and so the above expression bounding
	  $\beta-\alpha$ is minimized at $p_2^*=1/\sqrt{2}$,
	  giving $\beta-\alpha\leq c<1$. 
	  Applying Theorem \ref{tthm}, we find that $p_1^*=1/\sqrt{2}$ guarantees
	  a nonnegative return to player 1.
	  Similarly, substituting $p_2^*=1/\sqrt{2}$ gives return
	  $$
	  \alpha=\begin{cases}
		  0, & p_1^*\leq 1/\sqrt{2},\\
		  2(1/\sqrt{2}-p_1^*)((p_1^*)^2+ p_1^*/\sqrt{2} -1) & p_1^*\geq 1/\sqrt{2},
	  \end{cases}
	  $$
	  which is $\leq 0$.
	  Meanwhile $\beta(p_1^*,1/\sqrt{2},1/\sqrt{2})=2-\sqrt{2}<1$ independent of $p_1^*$, hence 
	  by another application of Theorem \ref{tthm}, we have that $p_2^*=1/\sqrt{2}$ guarantees a nonpositive
	  return.
  \end{proof}

  \br\label{fullrmk}
  Note that $\beta(1/\sqrt{2},p_2^*)<1$ does not hold for all $p_2^*$; that is, we have used the full power
  of Theorem \ref{tthm} to obtain this result.
  \er

  \br\label{3uniquermk}
  As in Remark \ref{nomixrmk} (2-player case), one may show that the pure strategy $p_1^*=1/\sqrt{2}$
  is the unique optimal strategy for player 1 against bloc strategies for players 2-3.
  Thus, for the general (nonbloc) game \emph{it is the unique candidate for a strategy guaranteeing
  nonnegative return}.
  \er

  \subsubsection{Nash equilibrium}\label{s:3Nash}
  Proposition \ref{3blocprop} includes the important consequence that 
  $$
  (p_1^*,p_2^*,p_3^*)=(1/\sqrt{2}, 1/\sqrt{2}, 1/\sqrt{2})
  $$
  is a {\bf symmetric Nash equilibrium}
  (see the discussion of the introduction) {\bf for the full} (unrestricted) {\bf 3-player game},
  i.e., $p_2^*=p_3^*=1/\sqrt{2}$ penalizes any deviation of $p_1^*$ from $p_1^*=1/\sqrt{2}$.

  \subsubsection{General case}\label{s:3gen}
  Next, we determine the best response function for the full (nonbloc) game.

  \begin{proposition}\label{3br}
For continuous 3-player Guts, the best response function is given by 
	  \be\label{R}
	  R(p_1^*):=\min_{p_2^*,p_3^*}\alpha(p^*_1,p^*_2,p^*_3)= \min \{\alpha_a(p_1^*), \alpha_b(p_1^*)\},
\ee
	  with respective values
	  \ba\label{3breq}
	  \alpha_a&=-\frac{1}{27}\Big((4(p_1^*)^2+6)^{\frac{3}{2}}+8(p_1^*)^3-36p_1^*\Big),\\
	  \alpha_b &=-\frac{2}{27}\Big((9(p_1^*)^2+3)^{\frac{3}{2}}-27p_1^*\Big) ,
	  \ea
achieved at $(p_2^*, p_3^*)=((\sqrt{4(p_1^*)^2+6}-2p_1^*)^{-1}, (\sqrt{4(p_1^*)^2+6}-2p_1^*)^{-1})$ and
	  $(p_2^*, p_3^*)=\Big(0, \sqrt{\frac{3(p_1^*)^2+1}{3}}\Big)$.
  \end{proposition}

  \begin{proof}
	  We consider the various cases in turn, minimizing on different regions of definition.
\\\\
  (1) $p_1^*<p_2^*<p_3^*$ and $p_1^*<p_3^*<p_2^*$\\
  Take $\frac{\partial \alpha }{\partial p_2^*} =6p_2^*p_3^*-4p_1^*p_3^*-1$ and $\frac{\partial \alpha}{\partial p_3^*}=3(p_2^*)^2+3(p_3^*)^2-4p_1^*p_2^*-1$, and second derivatives $\frac{\partial^2 \alpha}{\partial (p_2^*)^2}=\frac{\partial^2 \alpha}{\partial (p_3^*)^2}=6p_3^*$ and $\frac{\partial^2 \alpha}{\partial p_2^*p_3^*}=6p_2^*-4p_1^*$. We get $\lvert H \rvert=36(p_3^*)^2-(6p_2^*-4p_1^*)^2$ strictly positive with 
  $\frac{\partial^2 \alpha}{\partial (p_2^*)^2}\geq 0$, so we have a local minimum.
  	Solving $\frac{\partial \alpha}{\partial p_2^*}=0$ and $\frac{\partial \alpha}{\partial p_3^*}=0$ yields
  $p_2^*=p_3^*=(\sqrt{4(p_1^*)^2+6}-2p_1^*)^{-1}$, which gives $\alpha=-\frac{1}{27}((4(p_1^*)^2+6)^{\frac{3}{2}}+8(p_1^*)^3-36p_1^*)$.\\\\
  (2) $p_2^*<p_1^*<p_3^*$\\
  Take $\frac{\partial}{\partial p_2^*}\alpha_3=2p_1^*p_3^*-1$. $\frac{\partial}{\partial p_3^*}\alpha_3=2p_1^*p_2^*-3(p_1^*)^2+3(p_3^*)^2-1$. 
  The second derivatives yield the Hessian determinant is $-4(p_1^*)^2<0$, leading to the conclusion that the global minimum is a border solution. There are four border cases:\\
  (a) $p_2^*=0$, then $\alpha_3=2p_1^*-p_3^*+(p_3^*)^3-3(p_1^*)^2p_3^*$. Varying $p_3^*$ we get min$\alpha_3=2p_1^*-\sqrt{\frac{3(p_1^*)^2+1}{3}}+\frac{3(p_1^*)^2+1}{3}\sqrt{\frac{3(p_1^*)^2+1}{3}}-3(p_1^*)^2\sqrt{\frac{3(p_1^*)^2+1}{3}}$\\
  (b) $p_3^*=0$, then since $p_2^*<p_1^*<p_3^*$, $p_2^*=p_1^*=0$, and thus min$\alpha_3=0$.\\
  (c) $p_2^*=1$, then since $p_2^*<p_1^*<p_3^*$, we must have $p_1^*=p_3^*=1$, thus min$\alpha_3=0$.\\
  (d) $p_3^*=1$, then $\alpha_3=2p_1^*-p_2^*-3(p_1^*)^2+2p_1^*p_2^*=(2p_1^*-1)p_2^*-3(p_1^*)^2+2p_1^*$. If $p_1^*<\frac{1}{2},$ then $2p_1^*-1<0,$ min$\alpha_3 = p_1^*-(p_1^*)^2$. If $p_1^*\geq\frac{1}{2},$ then $2p_1^*-1\geq0,$ min$\alpha_3 = 2p_1^*-3(p_1^*)^2$. \\\\
  (3) $p_3^*<p_1^*<p_1^*$\\
  Take $\frac{\partial}{\partial p_2^*}\alpha_4=2p_1^*p_3^*-3(p_1^*)^2+3(p_2^*)^2-1=0$. $\frac{\partial}{\partial p_3^*}\alpha_4=  2p_1^*p_2^*-1=0$. 
  The second derivatives yield the Hessian determinant is also $-4(p_1^*)^2<0$, leading to the conclusion that the global minimum is a border solution. There are four border cases:\\
  (a) $p_2^*=0$, then  since $p_3^*<p_1^*<p_2^*$, $p_3^*=p_1^*=0$, and thus min$\alpha_4=0$.\\
  (b) $p_2^*=1$, then $\alpha_4=2p_1^*-p_3^*-3(p_1^*)^2+2p_1^*p_3^*=(2p_1^*-1)p_3^*-3(p_1^*)^2+2p_1^*$. If $p_1^*<\frac{1}{2},$ then $2p_1^*-1<0,$ min$\alpha_4= p_1^*-(p_1^*)^2$. If $p_1^*\geq\frac{1}{2},$ then $2p_1^*-1\geq0,$ min$\alpha_4=2p_1^*-3(p_1^*)^2$. \\
  (c) $p_3^*=0$, then $\alpha_4=2p_1^*-p_2*+(p_2^*)^3-3(p_1^*)^2p_2^*$. Varying $p_2^*$ we get min$\alpha_4=2p_1^*-\sqrt{\frac{3(p_1^*)^2+1}{3}}+\frac{3(p_1^*)^2+1}{3}\sqrt{\frac{3(p_1^*)^2+1}{3}}-3(p_1^*)^2\sqrt{\frac{3(p_1^*)^2+1}{3}}$\\
  (d) $p_3^*=1$, then since $p_3^*<p_1^*<p_2^*$, we must have $p_1^*=p_2^*=1$, thus min$\alpha_4=0$.\\\\
  (4) $p_2^*<p_3^*<p_1^*$ or $p_3^*<p_2^*<p_1^*$\\
  %Objective: minimize $\alpha(p_1^*, p_2^*, p_3^*)=2p_1^*-p_2^*-p_3^*-2(p_1^*)^3+2p_1^*p_2^*p_3^*$, 
  %with respect to $p_2^*$ and $p_3^*$, given that $p_2^*, p_3^*\leq p_1^*$.
  %\\\\
  Take $\frac{\partial}{\partial p_2^*}\alpha(p_1^*, p_2^*, p_3^*)=2p_1^*p_3^*-1$. 
  By symmetry, $\frac{\partial}{\partial p_3^*}\alpha(p_1^*, p_2^*, p_3^*)=2p_1^*p_2^*-1$. 
  This yields second derivatives $\frac{\partial^2}{\partial(p_2^*)^2}\alpha(p_1^*, p_2^*, p_3^*)=0$, 
  $\frac{\partial^2}{\partial(p_3^*)^2}\alpha(p_1^*, p_2^*, p_3^*)=0$, and 
  $\frac{\partial^2}{\partial p_2^*p_3^*}\alpha(p_1^*, p_2^*, p_3^*)=2p_1^*$. 
  The Hessian determinant is then calculated as $\frac{\partial^2\alpha}{\partial(p_2^*)^2}\frac{\partial^2\alpha}{\partial(p_2^*)^2}
  -(\frac{\partial^2\alpha}{\partial p_2^*p_3^*})^2= -4(p_1^*)^2<0$. As this is nonpositive, it is not an interior minimum, and so
  %this result is strictly nonpositive, leading to the conclusion that 
	  the global minimum is a border solution.
  %\\\\	
  There are four border cases: fix $p_2^*=0$ and vary $p_3^*\in[0,p_1^*]$, fix $p_2^*=p_1^*$ and vary $p_3^*\in[0,p_1^*]$,
  fix $p_3^*=0$ and vary $p_2^*\in[0,p_1^*]$, and fix $p_3^*=p_1^*$ and vary $p_2^*\in[0,p_1^*]$.
  By symmetry, the first and third cases are identical, as well as the second and fourth, leaving only two cases to consider: fix at 0,
  and fix at $p_1^*$.
  \\\\
  Fixing $p_3^*=0$ yields $\alpha(p_1^*, p_2^*, 0)=2p_1^*-p_2^*-2(p_1^*)^3$. From this, it is clear that the minimum will
  be achieved with $p_2^*=p_1^*$, giving $\alpha(p_1^*, p_1^*, 0)=p_1^*-2(p_1^*)^3$. The other case, with $p_3^*=p_1^*$,
  gives $\alpha(p_1^*, p_2^*, p_1^*)=p_1^*-p_2^*-2(p_1^*)^3+2(p_1^*)^2p_2^*$. 
	  Taking $\frac{\partial}{\partial p_2^*}\alpha(p_1^*, p_2^*, p_1^*)$ gives $2(p_1^*)^2-1$, which is only 0 if 
  $p_1^*=\frac{1}{\sqrt{2}}$, meaning it is impossible to conclude if this is a minimum or not. As the only other candidate is
  $\alpha(p_1^*, p_1^*, 0)=p_1^*-2(p_1^*)^3$, it must be the minimum of $\alpha$ for $p_2^*, p_3^*\leq p_1^*$.
  \end{proof}

From \eqref{R}-\eqref{3breq} we obtain a conclusion strikingly different than in the 2-player case.

  \begin{corollary}\label{3nopurecor}
	  For continuous 3-player Guts, \emph{there is no strategy, pure or mixed, 
	  guaranteeing a nonnegative return} for player 1.
  \end{corollary}

  \begin{proof}
	  By Remark \ref{3uniquermk}, the unique strategy, pure or mixed, 
	  guaranteeing nonnegative return against the restricted
	  class of bloc solutions, is the pure strategy $p_1^*=1/\sqrt{2}$.
  But, since $\alpha_a<0$ for $p_1^*\neq\frac{1}{\sqrt{2}}$,
  $\alpha_b<0$ for $p_1^*\notin[\approx0.248, \approx0.639]$, and$\frac{1}{\sqrt{2}}\notin[0.248, 0.639]$,
  we have $R(p_1^*)<0$ for any $p_1^*$, and so no pure strategy can guarantee nonnegative return, in particular
	  not the candidate $p_1^*=1/\sqrt{2}$.
  \end{proof}

  \br\label{con-conrmk}
  Functions $\alpha_j(p_1^*, p_2^*,p_3^*)$, $j=1,\dots,6$ as in the proof of Proposition
  \ref{3prop} are concave in $p_1^*$.
  Likewise, $\alpha_1$ and $\alpha_2$ are convex in $(p_2^*,p_3^*)$. However, none of $\alpha_3$, $\alpha_4$,
  $\alpha_5$, $\alpha_4$ are convex in $(p_2^*,p_3^*)$; hence the minimax theorem does not apply, and optimal
  pure solutions are not guaranteed.
  \er

  \subsection{A winning strategy for players 2-3}\label{s:3win}
  The abstract result of Corollary \ref{3nopurecor} is not completely satisfying, relying on the
  subtle observation of Remark \ref{nomixrmk} rather than direct computation.
  More important, though it shows that players 2-3 can force a losing outcome for player 1,
  this does not imply that they can force a winning outcome for themselves.
  For, recall (Section \ref{s:value}) that unbounded generalized recursive games are not necessarily zero-sum.
  We complete our treatment of the 3-player case by exhibiting a rather explicit winning strategy for
  players 2-3, combining just 2 pure strategies.

  \begin{lemma}\label{L1}
	  For some fixed $C_0,C_1,C_2,\delta >0$, and $\epsilon>0$ sufficiently small,
  the strategy A. $p_2^*=p_3^*= 1/\sqrt{2} -\epsilon$ satisfies
  $\alpha(p_1,p_2^*,p_3^*)< C_1 \epsilon2$ for $p_1\in [1/\sqrt{2}- C_0\epsilon, 1/\sqrt{2} + C_0\epsilon]$,
  $\alpha(p_1,p_2^*,p_3^*)<0$ for $p_1\not \in (1/\sqrt{2}- C_0\epsilon, 1/\sqrt{2} + C_0\epsilon)$,
  and
  $\alpha(p_1,p_2^*,p_3^*)\leq - \epsilon/C_2$ for $p_1\in [1/\sqrt{2}- \delta, 1/\sqrt{2} + \delta]$.
  \end{lemma}

  \begin{proof}
	  By the bloc alpha formula, for $p_1\leq p_2^*$, we have
  $\alpha(p_1,p_2^*,p_2^*)= 2(p_1-p_2^*)(1-2(p_2^*)^2)$, i.e., a linear function of $p_1$ vanishing at $p_2^*$,
  with slope $2(1-2(p_2^*)^2>\epsilon/c$ for some fixed $c>0$, any $\epsilon>0$.
  Thus, $\alpha<-\eps/C_2$ for 
  for $p_1\leq 1/\sqrt{2}- \delta$ for $C_2$ sufficiently large, and $\alpha<0$ for
  $p_1\in [p_2^*-\delta, p_2^*-C_0\epsilon]$, while $\alpha\leq C_1\epsilon^2$ for 
  $p_1\in [p_2^*-C_0\epsilon, p_2^*]$ for $C_1$ sufficiently large.
  Similarly, for $p_1\geq p_2^*$, we have
  $$
	  \begin{aligned}
		  \alpha(p_1,p_2^*,p_2^*)&= 2(p_1-p_2^*)(1-(p_1^*)^2 - p_1^*p_2^*)\\
		  &\leq
  2(p_1-p_2^*)(1-(p_1^*)^2 - p_1^*p_2^*)\leq
  2(p_1-p_2^*)(1-(p_1^*)^2 - (p_2^*)^2).
	  \end{aligned}
  $$

  Observing that $(p_2^*)^2\geq 1/2 - \epsilon/c$ for 
  $p_2^*=1/\sqrt{2}-\epsilon$ and $\epsilon$ sufficiently small, this gives
  $$
  \alpha(p_1,p_2^*,p_2^*)\leq
  (p_1-p_2^*)(1-(p_1^*)^2 + 2\epsilon/c),
  $$
  from which the estimates readily follow.
  \end{proof}

  \begin{lemma}\label{L2}
For some $\delta, \sigma >0$, and $C_3>0$,
  the strategy B. $(p_2^*, p_3^*)= (0, 1/\sqrt{2} + \delta )$ 
  satisfies $\alpha(p_1,p_2^*,p_3^*)<-1/C_3$ 
  for $p_1\in [1/\sqrt{2}- \sigma,  1/\sqrt{2} + \sigma]$.
  \end{lemma}

 \begin{proof} 
Assuming $\sigma<\delta$, for $p_1\in [1/\sqrt{2}- \sigma,  1/\sqrt{2} + \sigma]$ we are in case (2)
	 of the proof of Proposition \ref{3prop}, so that $\alpha $ is given by
$$
\alpha_3(p_1^*,p_2^*, p_3^*)= 2p_1^*-p_2^*-p_3^*+(p_3^*)^3-3(p_1^*)^2p_3^*+2p_1^*p_2^*p_3^*
$$
or $2p_1^*-p_3^*+(p_3^*)^3-3(p_1^*)^2p_3^*.  $
For $p_1^*=p_3^*=1/\sqrt{2}$, this is $0$. Taking the derivative with respect to $\delta$ of
$\alpha(1/\sqrt{2}, 0, 1/\sqrt{2}+\delta)$, or in other words, the derivative of $\alpha$ with respect to $p_3$,
we obtain 
$\alpha_3'(0 <p_1^*<p_3^*)= -1+3(p_3^*)^2-3(p_1^*)^2$, which at $\delta=0$ is
$-1 $.  Thus, for $\delta>0$ sufficiently small, we have $\alpha(1/\sqrt{2},0,1/\sqrt{2}+\delta)<0$.
By continuity, we then have for $\sigma < \delta$ sufficiently small, that
$\alpha(p_1,0,1/\sqrt{2}+\delta)<0$ for $p_1\in [1/\sqrt{2}-\sigma, 1/\sqrt{2}+\sigma]$ as claimed. 
  \end{proof}

\begin{corollary}\label{ABcor}
	For some fixed $C_0,C_1,C_2,C_3, \delta, \sigma >0$,
	$C>0$ sufficiently large, and $\epsilon=\epsilon(C)>0$ sufficiently small,
  the mixed strategy $ (1-C\epsilon^2)A+ C\epsilon^2 B$ returns $\alpha<0$, $\beta\leq \beta_0<1$ 
	for all $p_1^*\in [0,1]$, and therefore \emph{is a winning strategy for players 2-3}.
\end{corollary}

  \begin{proof}
	  By Lemmas \ref{L1}-\ref{L2}, $C\epsilon^2 B$ gives $\alpha \leq C_4 C\epsilon^2$ for 
  $p_1\not\in [1/\sqrt{2}-\delta, 1/\sqrt{2}+\delta]$, while $(1-C\epsilon^2) A$  gives $\leq -\epsilon/C_2$ for
  $\epsilon>0$ sufficiently small. The sum is thus 
  $\leq C_4 C\epsilon^2- \epsilon/C_2 <0$ for
  $\epsilon(C) $ sufficiently small.
  Both A and B return $\alpha<0$ for $p_1$ in $[1/\sqrt{2}-\delta,1/\sqrt{2}+\delta]$ but not in 
  $[1/\sqrt{2}-C_0\epsilon,1/\sqrt{2}+C_0\epsilon]$.
  Finally, on $[1/\sqrt{2}-C_0\epsilon,1/\sqrt{2}+C_0\epsilon]$, $(1-C\epsilon^2)A$ returns $\leq 2C_1\epsilon^2$
  for $\epsilon$ sufficiently small, while $C\epsilon^2 B$ returns $<-C\epsilon^2 /C_3$, hence the sum gives
  $\alpha \leq 2C_1\epsilon^2 -C\epsilon^2 /C_3< 0$ for $C$ sufficiently large.

	 Finally, recall from the proof of Proposition \ref{3blocprop}
that $\beta(p_1^*,1/\sqrt{2},1/\sqrt{2})=2-\sqrt{2}<1$ independent of $p_1^*$, whence by continuity
$\beta(p_1^*,1/\sqrt{2}-\eps,1/\sqrt{2}-\eps)=2-\sqrt{2}\leq \beta_1<1$ independent of $p_1^*$.
	  Noting that $\beta$ is at least bounded for strategy B, and that B is weighted by $O(\eps^2)\to 0$
	  in the proposed mixed strategy, we thus have that $\beta\leq \beta_0<1$ for $\eps$ sufficiently
	  small.
	  Thus, $\alpha<0$ and $\beta\leq \beta_0<1$, hence,
	  by Theorem \ref{tthm}, the mixed A/B strategy gives a strictly positive return for players 2-3.
  \end{proof}

  \begin{example}
	  An explicit example of a winning player 2-3 strategy of the mixed type A/B described in Corollary \ref{ABcor}
	  is given by the choices $\eps=0.04$, $\delta=0.137$, and $C=(17/16)\times 10^2$.
	  As depicted in Figure \ref{negativefig}, the best case return for player 1 is 
	  $\lessapprox -0.04<0$, or $\approx -0.4\%$ of ante $1.0$.
  \end{example}

  \begin{figure}
        \centering
        \begin{subfigure}[b]{0.45\textwidth}
		 (a) \includegraphics[scale=0.25]{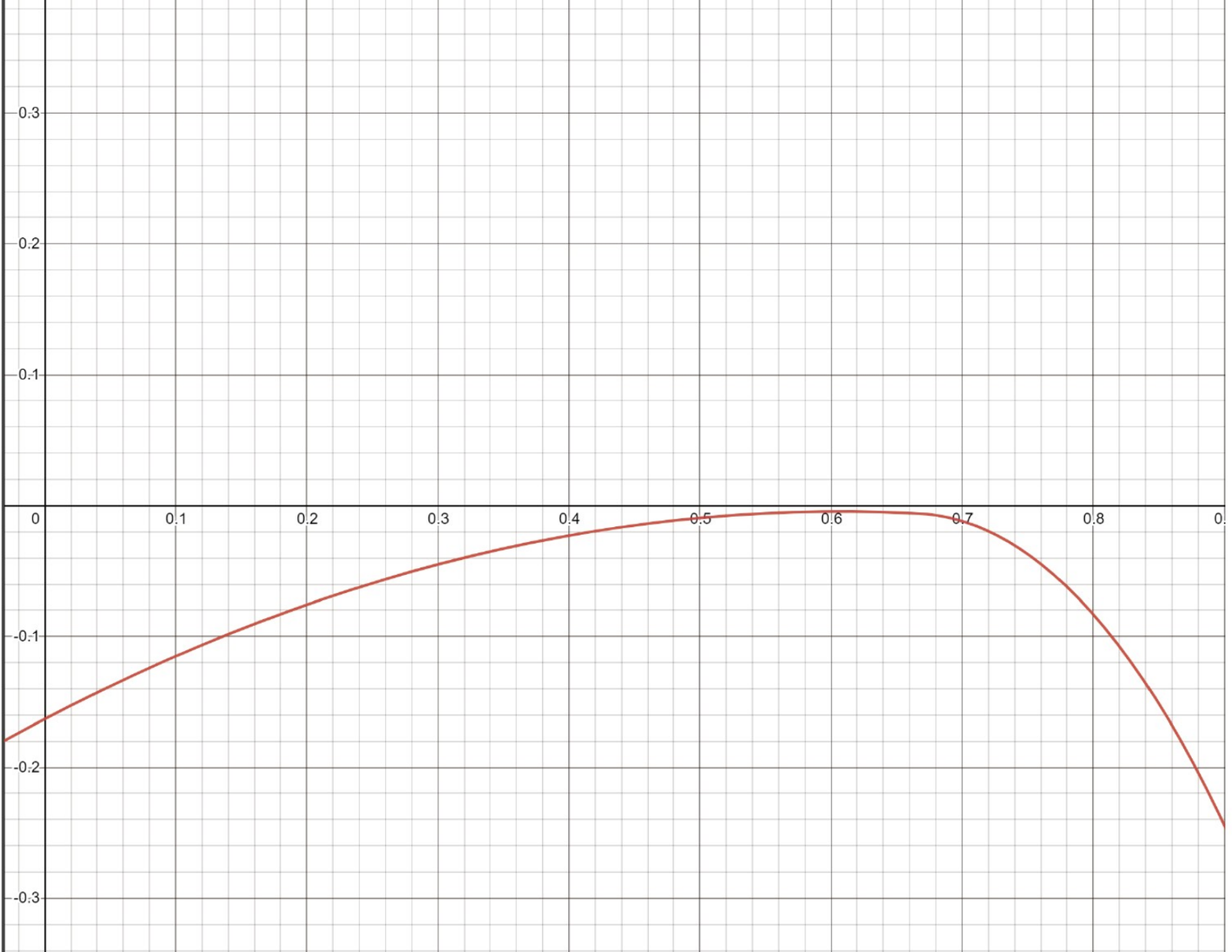}
		 \end{subfigure} 
		 \quad
		 \begin{subfigure}[b]{0.45\textwidth}
			 (b) \includegraphics[scale=0.25]{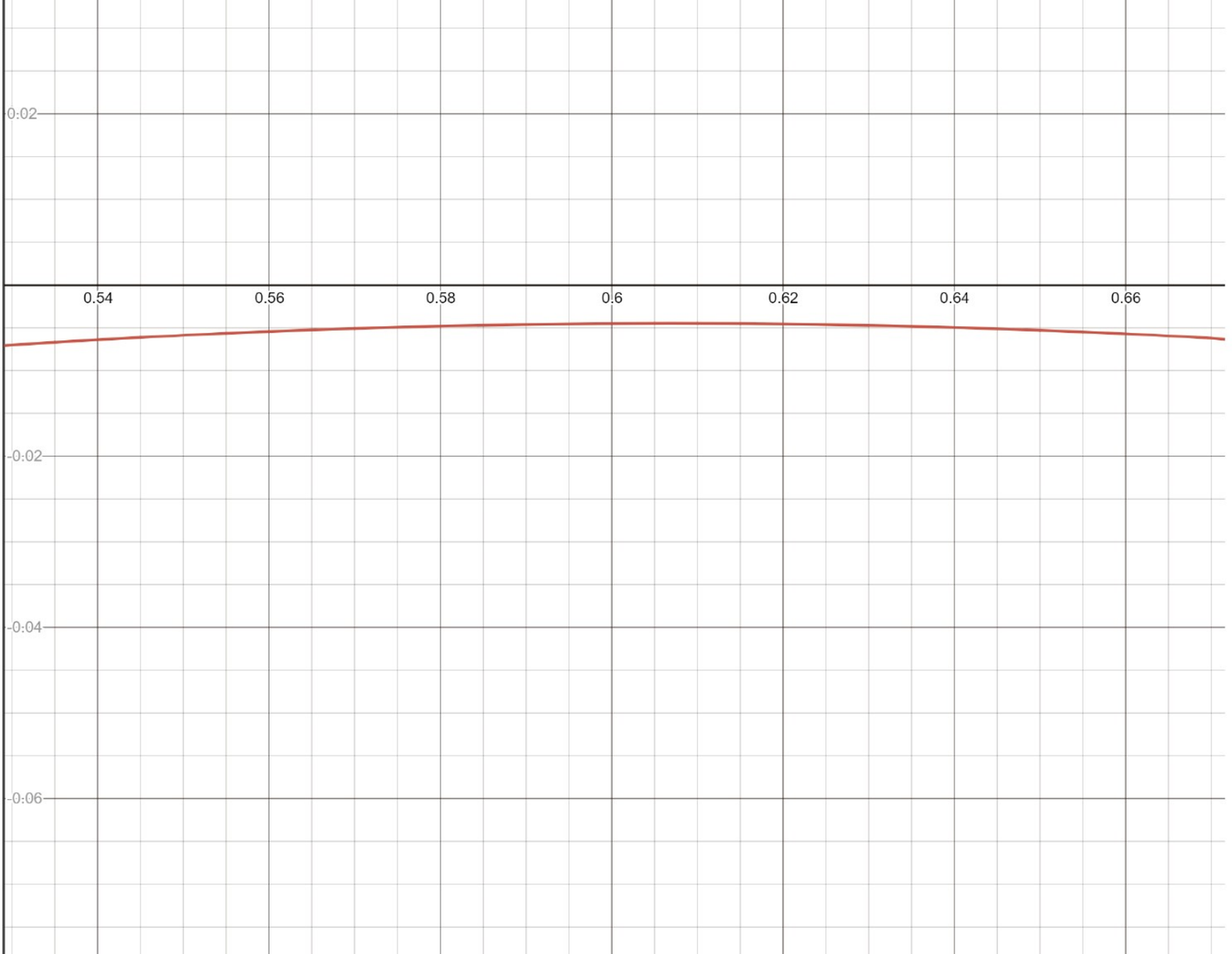}  
		 \end{subfigure}
		 	\caption{
				Graphs plotting expected return $\alpha$ vs. $p_1^*$ for mixed player 2-3 strategy A/B
				of Corollary \ref{ABcor}, with $\eps=0.04$, $\delta=0.137$, and $C=(17/16)\times 10^2$:
				a) full scale. b) blowup near $p_1^*=1/\sqrt{2}$. Maximum return for player 1
				is $\approx -0.004$, or approximately $0.4 \%$ of the initial outlay (ante) $1.0$.
				}
	\label{negativefig}
  \end{figure}

  \subsubsection{Strong Nash equilibrium}\label{s:3sNash}
  Corollary \ref{ABcor} includes the information that the symmetric Nash equilibrium
  $ (p_1^*,p_2^*,p_3^*)=(1/\sqrt{2}, 1/\sqrt{2}, 1/\sqrt{2}) $
  {\bf is not a Strong Nash equilibrium},
  i.e., for $p_1^*=1/\sqrt{2}$ there exist deviations 
  from equilibrium strategy $(1/\sqrt{2},1/\sqrt{2})$ 
  both forcing a negative return for player 1 and improving the joint return for players 2-3.

\section{Analysis of the continuous $n$-player game}\label{s:n}
Finally we examine the $n$-player game, $n\geq 4$, again restricting to the continuous case.
Here, we do not attempt to determine the full payoff and best response functions, but only, guided by
our analyses of the previous case, to show using judiciously chosen special cases that the broad outlines
of behavior are the same as in the case $n=3$: namely, there exists an optimal pure strategy for player 1 against
bloc strategies for players 2-$n$, guaranteeing nonnegative return, but there exists no strategy, pure or mixed, 
guaranteeing nonnegative return agains general coalition strategies of players 2-$n$.  More, there
exists a winning strategy for players 2-$n$, returning to them a strictly positive expected value.
Thus, \emph{there is a symmetric Nash equilibrium but it is not strong}.

\subsection{Bloc case: necessity}\label{s:nbloc}
We first note that, by consideration of the special, bloc strategy case,
we may show by elimination that the only possible optimal
``go/no-go'' strategy for the $n$-player game consists of $p_1^*=(1/2)^{1/(n-1)}$. 
The key is the following miraculously simple formula.

\begin{lemma}\label{derivlem}
	The derivative of $\alpha(p_*+h, p_*, \dots, p_*)$ with respect to $h$ at $h=0$ is given by
	\be\label{hderiv}
	(d/dh)\alpha(p_*+h, p_*, \dots, p_*)|_{h=0}= (n-1)(1-2p_*^{n-1}).
	\ee
\end{lemma}

\begin{proof}
The strategy vectors $(p_*+h, p_*, \dots, p_*)$ and $(p_*, p_*, \dots, p_*)$ prescribe identical actions for all
	players, except when $p_1\in [p_*,p_*+h]$, as occurs with probability $h$.  Conditioning on this event,
	we find that is is sufficient to show that the difference in conditional expected payoffs is
	$(n-1)(1-2p_*^{n-1})+ O(h)$. Without loss of generality take $h>0$; a symmetric computation suffices
	for $h<0$.
	Noting that $p_j\in [p_*,p_*+h]$ for some $j=2,\dots, n$ occurs with probability $O(h)$, we may ignore this
	case as being accounted for in the $O(h)$ error term. In the limit, it is thus equivalent to taking
	$p_1=p_*$ and each $p_j$, $2\leq j\leq n$ randomly in $[0,p_*]$ or $[p_*,1]$, and computing the difference in
	expected payoffs for player 1 dropping ($h>0$) and holding ($h=0$).

	In the case that all of $p_2,\dots,p_n$ lie in $[0,p_*]$, the difference is $0- (n-1)=-(n-1)$.
	In all other cases, $n-1-m$ of $p_2,\dots,p_n$ lie in $[0,p_*]$ and $m$ in $[p_*,1]$, $m\geq 1$.
	If player 1 drops, he loses his initial ante $1$, but gains back gratis the ante $m-1$ to the
	replayed game with stakes multiplied by $m-1$.  If player 1 holds, he loses his ante $1$ plus
	the amount of the pot, $n$, but gains back gratis the ante $m$ of the replayed game.
	Thus, the difference in payoff is $(m-2)- (-n-1+m)= +(n-1)$.

	Summing over all cases, we find that the difference in payoffs is $(n-1)$ times $1$ (the total of all probabilities) minus the difference between the payoff $-(n-1)$ of the first event and the payoffs $+(n-1)$ of all other events,
	times the probability $p_*^{n-1}$ of the first event.  This gives $(n-1)$ times
	$(1- 2p_*^{n-1})$ as claimed, verifying \eqref{hderiv}.
\end{proof}

\begin{corollary}\label{elimcor}
	For $n$-player guts in the bloc strategy case, 
	the optimal pure strategy for player 1 if it exists is given by $p_1^*=(1/2)^{1/(n-1)}$.
\end{corollary}

\begin{proof}
	From \eqref{hderiv} we find (by symmetry) that
	$(d/dh)\alpha(p, p_*+h, p_*, \dots, p_*)|_{h=0}= -(1-2p_*^{n-1})$.
	Since $\alpha(p_*, \dots, p_*)=0$ (symmetry), $\alpha(p_*, \cdot,\dots, \cdot)\geq 0$
	for all $p_2^*,\dots, p_n^*$ only if 
	$(d/dh)\alpha(p, p_*+h, p_*, \dots, p_*)|_{h=0}=0$, or $ 1= 2p_*^{n-1}$.
	Thus, $p_1^*=p_*$ is optimal only if $p_1^*= (1/2)^{1/(n-1)}$.
\end{proof}

\br\label{nconconrmk}
Similar reasoning shows that the bloc payoff function $\alpha(p_1^*,p_2^*,\dots, p_2^*)$ is concave-convex,
whence we may obtain existence of a pure strategy returning $\alpha \geq 0$ by the minimax theorem, Theorem \ref{mmthm},
which must therefore (by Corollary \ref{elimcor}) be $p_1^*=(1/2)^{1/(n-1)}$.
We will establish this result instead by direct computation in the following section.
\er

\subsection{Bloc case: sufficiency}\label{s:nblocsuff}
By much the same argument we obtain the following global result.

\begin{proposition}\label{blocprop}
	For $p_1^*\leq p_2^*$,
	\be\label{formeq}
	\alpha(p_1^*, p_2^*, \dots, p_2^*)= (n-1)(p_2^*-p_1^*)\Big(2(p_2^*)^{n-1}-1\Big)
	\ee
	is \emph{linear in $p_1^*$}.  For $p_1^*> p_2^*$ on the other hand, 
	\be\label{formineq}
	(n-1)(p_2^*-p_1^*)(2(p_1^*)^{n-1}-1)+ \delta =
	\alpha(p_1^*, p_2^*, \dots, p_2^*)< (n-1)(p_2^*-p_1^*)\Big(2(p_2^*)^{n-1}-1\Big),
	\ee
	where corrector $\delta$ is $0$ for $n=2$ and for $n\geq 3$ is $>0$, satisfying
	\be\label{corrector}
	\delta \geq (n-1)(n-2)\Big(\frac{ (p_1^*-p_2^*)^{2}(p_2^*)^{n-2} + 2(p_1^*-p_2^*)^{n} } {(p_1^*)^{n-1} }\Big).
	\ee
\end{proposition}

\begin{proof}
	For $h\leq 0$, the proof of Lemma \ref{hderiv} may be applied without change to give 
	$$
	\alpha(p_*+h, p_*, \dots, p_*)= h (n-1)(1-2p_*^{n-1}),
	$$
	since the assumptions made in throwing out $O(h)$ error terms are in this case exact.
	This proves \eqref{formeq} in the case $p_1^*=p_*+h\leq p_*=p_2^*$.
	In the case $h\leq 0$, on the other hand,
	calculating the difference in expected return for player 1 when dropping instead of
	holding, we find that the $O(h)$ error terms are of negative sign (corresponding to cases when player 1 would
	have won had he held).  This proves the righthand inequality in \eqref{formineq} in the case
	$p_1^*=p_*+h\geq p_*=p_2^*$.
	To establish the lefthand inequality, we have only to observe that, performing the same estimate as in 
	the proof of Lemma \ref{hderiv}, but considering cases that $p_2,\dots,p_n$ are in $[0, p_1^*]$
	vs. $[p_1^*,1]$ instead of $[0, p_*]$ vs $[p_*,1]$, we obtain instead
	$$
	\alpha(p_1^*, p_2^*, \dots, p_2^*)= (n-1)\Big(p_2^*-p_1^*)(2(p_1^*)^{n-1}-1\Big) + O(h),
	$$
	where now the $O(h)$ errors vanish except in case $p_1\in [0,p_1^*]$.
	Calculating the difference in expected return for player 1 when holding instead of dropping, 
	we find in this case that the $O(h)$ error terms are of positive sign 
	(corresponding to cases when player 1 would have lost had he dropped).

	This establishes \eqref{formineq} with $\delta>0$. To obtain the more precise estimate \eqref{corrector},
	we examine the cases when $m=0,\dots,n-1$ of the opposing players' $p_j$ lie in $[p_2^*,p_1^*]$
	and adjust our estimate for the difference in expected return to player 1 from dropping vs. holding.
	Denote the difference for each value $m$ as $\delta_m$, so that
	\be\label{deltasum}
	\delta =(p_1^*-p_2^*)\sum_m \delta m \rho_m,
	\ee
	is the total correction in our estimate for $\alpha$, where 
	\be\label{rho}
	\rho_m:= C(n-1,m)(p_2^*)^{n-1-m}(p_1^*-p_2^*)^m (p_1^*)^{n-1}
	\ee
	is the conditional probability that $m$ of the opposing $p_j$ lie in $[p_2^*,p_1^*]$.

	Our lower bound was to assume incorrectly that all of players $2,\dots,n$ drop, so that player 1 would
	win $0$ when dropping and $n-1$ when holding, for a net change of $-(n-1)$.
	This is exact for $m=0$, i.e., $\delta_0=0$, but must be corrected for $1\leq m\leq n-1$.
	For $m=1$, player 1 wins $-1$ when dropping; when holding, it is a fair game between player 1 and the other
	player with $p_j\in [p_2^*,p_1^*]$, and so they win $0$, for actual difference of $-1$.  Thus,
	$\delta_1= (-1) - (-(n-1))= n-2$, which is $0$ for $n=2$ and $>0$ for $n\geq 3$. 

	Consider now the general case $m\geq 2$, $n\geq 3$.
	If player 1 drops then there is a fair game between the $m$ players with 
	$p_j\in [p_2^*,p_1^*]$, hence a replay with stakes $m-1$ times higher;
	player 1 thus loses their original ante but gains a ``virtual ante'' of $m-1$ in the repeated, higher-stake
	game, for net win of $m-2$.
	If player 1 holds, there is a fair game between the $m+1$ players consisting of player 1 and
	the other players with $p_j\in [p_2^*,p_1^*]$, hence a replay with stakes $m$ times higher.
	Thus, player 1 wins $n$ with probability $1/(m+1)$ and $-n$ with probability $m/(m+1)$, but gains
	a virtual ante of $m-1$ in the repeated $m$ times higher-stake game, above their original ante of $1$.
	Thus, their expected net payoff when holding is $-(m-1)n/m + (m-1)$, for an actual net difference in 
	expectation between dropping and holding of $(m-2)- (m-1)+ n(m-1)/m=  n(m-1)/m -1.$

	Collecting information, we find that the $m$th corrector is 
	$$
	\delta_m =  n(m-1)/m -1 - (-(n-1))= (n-2)+ n(m-1)/m \geq 0.
	$$
	In the extreme case $m=n-1$, we have in particular $\delta_{n-1}=(n-2)(1+ n/(n-1))\geq 2(n-2)$.
	Plugging our estimates for $\delta_1$ and $\delta_{n-1}$ into \eqref{deltasum}-\eqref{rho},
	gives finally \eqref{corrector}, completing the proof.
\end{proof}

\br\label{3ineqrmk}
Formulae \eqref{formeq}--\eqref{formineq} are exact for $n=2$ (cf. \eqref{2alpha}).
From \eqref{3alpha}, we have for $n=3$ and $p_2^*\leq p_1^*$ the exact formula
	$
	\alpha(p_1^*,p_2^*,p_2^*)= 2(p_1^*-p_2^*)(1-2p_1^*\big(p_2^+ p_1^*)\big),
	$
	yielding explicit error bounds illustrating \eqref{formineq} of
	$
	\alpha(p_1^*,p_2^*,p_2^*)- 2(p_2^*-p_1^*)\Big(2(p_1^*)^2-1\Big)
	= 2p_1^*(p_2^*-p_1^*)^2
	$
	and
	$$
	\alpha(p_1^*,p_2^*,p_2^*)- 2(p_2^*-p_1^*)\Big(2(p_2^*)^2-1\Big)= -(2p_1^*+ 4p_2^*) (p_1^*-p_2^*)^2.
	$$
	From the proof of Proposition \eqref{blocprop}, we can extract if needed an exact error series representation for
	$\delta$ in powers of $(p_1^*-p_2^*)^r$, $r=1,\dots, n$. 
\er

\begin{corollary}\label{extcor}
	For $n$-player guts, with $p_1^*=(1/2)^{1/(n-1)}$, $\alpha(p_1^*, p_2^*, \dots, p_2^*)\geq 0$.
\end{corollary}

\begin{proof}
	Taking $p_1^*=(1/2)^{1/(n-1)}$ in Proposition \ref{blocprop}, we find for $p_2^*\geq p_1^*$, 
	by \eqref{formeq}, that
	$$
	\alpha(p_1^*, p_2^*, \dots, p_2^*)= (n-1)(p_2^*-p_1^*)\Big(2(p_2^*)^{n-1}-1\Big)\geq 0.
	$$
	For $p_2^*< p_1^*$, on the other hand, the lefthand inquality of \eqref{formineq} gives
	$$
	\alpha(p_1^*, p_2^*, \dots, p_2^*) > (n-1)(p_2^*-p_1^*)\Big(2(p_1^*)^{n-1}-1\Big)= 0.
	$$
\end{proof}

\begin{lemma}\label{betalem}
	For $n$-player Guts with $p_1^*=(1/2)^{1/(n-1)}$, 
	\be\label{lb}
	\beta(p_1^*, p_2^*, \dots, p_2^*)\leq  2\ln2 -1/2 < 0.9
	\ee
	for $p_2^*\geq p_1^*$, while for $p_2^*\leq p_1^*$,
	\be\label{b-a}
	\beta (p_1^*, p_2^*, \dots, p_2^*) - \alpha(p_1^*, p_2^*, \dots, p_2^*) 
	\equiv \beta(p_1^*,p_1^*, \dots, p_1^*) < 0.9.
	\ee
\end{lemma}`

\begin{proof}
	We begin by establishing $\beta(p_1^*, p_1^*, \dots, p_1^*)<1$.
	We subdivide into subcases that $0\leq m\leq n$ players hold, occurring with probabilities 
	$\mu_m=C(n,m)(p_1)^{n-m}(1-p_1^*)^{m}$.
	In cases $m\geq 1$, stakes increase by factor $(m-1)$, while in the special case $m=0$ that all
	players drop, the game is repeated with multiplication factor $1$.
	Thus, 
	\be\label{betasum}
	\beta= \mu_0 + \sum_{j=2}^{n} (m-1) \mu_m.
	\ee

	Evidently, 
	$$
	C(n,0)\mu_0= (p_1^*)^n = p_1^*/2 < 0.5.
	$$
From $(1-p_1^*)= 1- e^{-(1/(n-1)\ln (2)}$ and $ \frac{1}{1+x}\geq e^{-x}\geq 1-x$ for $x\geq 0$, we obtain 
	\be\label{mu0}
	\frac{\ln 2}{n} \leq \frac{\ln 2}{(n-1)+\ln 2}\leq (1-p_1^*)\leq \frac{\ln 2}{n-1}.
	\ee
	Meanwhile, 
	$$
	\begin{aligned}
		C(n,m)\mu_m&= \frac{n!}{m!(n-m)!}(1-p_1^*)^m (p_1^*)^{n-m}\\
		&\leq \frac{n!}{m!(n-m)!} \frac{(\ln 2)^m}{n^m} 
		\leq \frac{(\ln 2)^m}{m!}, 
	\end{aligned}
	$$
	hence 
	$ \sum_{m=2}^n (m-1)C(n,m)\mu_m \leq \sum_{j=2}^n (m-1) \frac{ (\ln 2)^m}{m!}$.
	Observing that 
	$$
	\sum_{j=2}^n (m-1) \frac{ x^m}{m!} = x^2 (d/dx)\frac{e^x-1}{x}
	= xe^x-e^x+1 ,
	$$
	we have setting $x=\ln 2$ the estimate
	\be\label{mu2}
	 \sum_{m=2}^n (m-1)C(n,m)\mu_m \leq
	 2 \ln 2-2+1 = 2\ln2 -1= .3809\dots < 0.4.
	\ee

	Combining \eqref{betasum}, \eqref{mu0}, and \eqref{mu2}, we thus have
	\be\label{bineq}
	\beta= \mu_0 + \sum_{j=2}^{n} (m-1) \mu_m < 0.9
	\ee
	for $p_2^*=p_1^*$.
	Indeed, one may check that all estimates carried out for this case remain valid also for $p_2^*\geq p_1^*$,
	hence, by $\alpha\geq 0$, \eqref{bineq} holds also for $p_2^*\geq p_1^*$, giving
	$\beta \leq  1-\sigma$ for $\sigma>0$, verifying \eqref{lb}.

	We next observe for $p_2^*\leq p_1^*$ that 
	$$
	(\alpha- \beta)(p_1^*, p_2^*, \dots, p_2^*) \equiv (\alpha- \beta)(p_1^*, p_1^*, \dots, p_1^*),
	$$
	giving \eqref{b-a}.
	For, any changes between left- and righthand sides are due to different treatments by players
	2-$n$ of the case that $p_j\in [p_2^*,p_1^*]$. Since player 1 drops for $p_1\in [0,p_1^*]$
	and holds for $p_1\in [p_1^*,1]$, player one's chances of winning the pot are not affected by whether
	or not such players hold, since in the first case player 1 has no chance regardless, and in the
	second case player 1 automatically has a higher hand than they do.  Thus, the changes in $\alpha$
	and $\beta$ are equal and opposite, both coming from changes in the size of the stakes: the first by
	a corresponding increase in ``virtual ante,'' and the second directly by change in the multiplication factor
	for the stakes.
	This is clear when there are one or more players with $p_j\in [p_1^*,1]$, or no players 
	with $p_j\in [p_1^*,1]$ and none of players 2-9 with $p_j\in [p_2^*,p_1^*]$.
	In the special case that there are no $p_j\in [p_1^*,1]$ and $r$ of players 2-9 
	with $p_j\in [p_2^*,p_1^*]$, player 1 drops, and $r$ of players 2-9 change from drop to hold;
	thus, the contribution to $\alpha$ changes from $0$ to $r-2$ and the contribution to $\beta$
	(through size of the stakes) changes from $1$ to $r-1$.  Thus, the contribution to $\alpha-\beta$
	changes from $(0-1)=-1$ to $(r-2)-(r-1)=-1$, again with net change zero. Summing against probabilities,
	this gives a net change in $\alpha-\beta$ of zero, as claimed.
\end{proof}

\br\label{improvermk}
A nicer estimate is to notice that $\sum_{m=0}^n m C(n,m)(1-p)^m p^{n-m}= n(1-p)$, the expected number of successes 
%(prob. $(1-p)$) 
in $n$ Bernoulli trials, while $\sum_{m=0}^n  C(n,m)(1-p)^m p^{n-m}= 1$.
This gives an exact formula of
$\beta(p,\dots,p)= n(1-p)-1+2p^n,$
or, in the case $p^{n-1}=1/2$, 
$$
\beta(p_*,\dots,p_*)= (n-1)(1-p_*)\leq \ln 2\leq 0.7,
$$
Moreover, $\beta(p,\dots,p)$ is monotone increasing in $p$ for 
$-n+ 2np^{n-1}= n(2p^{n-1}-1)>0$, or $p\geq (1/2)^{n-1}$,
so that we could obtain the result from $\beta(1,\dots,1)=1$, {\it or} from $\beta(0,\dots,0)=1$.  
It seems that the minimum value of $\beta(p,\dots,p)$ is in fact obtained
at $p=(1/2)^{n-1}$, interesting. 
The argument for $p_2^*\geq p_*$ is more subtle, requiring that $\sum_{m=1}^n (m-1)\mu_m<1/2$.
To get this estimate, we can use 
$$
\sum_{m-1}^n (m-1)\mu_m= \sum_0^n m\mu_m - \sum_1^n \mu_m= \sum_0^n m\mu_m -\sum_0^n \mu_m + \mu_0
=
n(1-p) - 1 + p^n,
$$
or, for $p=p_*$, 
$(n-1)(1-p) -  p/2= (n-1/2)(1-p) + 1/2 \leq \frac{(n-1/2)}{(n-1)}\ln(2)- p/2\leq (5/4) 0.7 -0.5\leq  0.4$,
the same as in the previous proof.
\er

\begin{corollary}\label{blocprop2}
	For $n$-player Guts, the strategy $p_1^*=(1/2)^{1/(n-1)}$ 
	satisfies termination condition \eqref{tcond} with respect to bloc strategies $p_2^*=p_3^*=\dots=p_n^*$.
	Thus, it is optimal in the strong sense that it
	can force with probability $1$ termination of the game with an expected total return $\geq 0$.
\end{corollary}

\begin{proof}
	Rewriting \eqref{tcond} (for $t=1$) as $\beta-\alpha \leq 1-\eps$ for $\eps>0$,
	and recalling that $\alpha\geq 0$, we see by \eqref{lb}-\eqref{b-a} that \eqref{tcond}
	is satisfied for n-player guts with $p_1^*=(1/2)^{1/(n-1)}$ and bloc strategies $p_2^*=p_3^*=\cdots=p_n^*$.
	By Theorem \ref{tthm}, the result follows.
\end{proof}

\br\label{nomixrmkn}
Using the fact (Remark \ref{nconconrmk}) that the bloc payoff function is concave in $p_1^*$,
we may conclude as in Remark \ref{nomixrmk} that the pure strategy $p_1^*$ is the unique strategy,
pure or mixed, returning $\alpha\geq 0$ against bloc strategies $p_2^*=\dots = p_n^*$.
\er

\subsection{Nonbloc coalition I: negative return for player 1}\label{s:nonbloc}

We now show, by adapting the results already established against bloc strategies, that $p_1^*$ is {\it not}
optimal against non-bloc strategies, by exhibiting a particular class that force negative outcome.

\begin{proposition}\label{nonblocprop}
	For the optimal bloc-strategy $p_1^*=(1/2)^{1/(n-1)}$, $\alpha(p_1^*,0, p_1^*,\dots, p_1^*)=0$,
	while 
	\be\label{keynest}
	(d/dh)^+\alpha(p_1^*,0, p_1^*+h, \dots p_1^*+h)= -(n-2).
	\ee
\end{proposition}

\begin{proof}
	By \eqref{formeq}, we have
	$\alpha(0, p_1^*, \dots, p_1^*)= (n-1)(p_1^*-0)\Big(2(p_1^*)^{n-1}-1\Big)=0$,
	whence, by symmetry, 
	$$
	\alpha(p_1^*, 0, p_1^*, \dots, p_1)=-\big(1/(n-1)\big) \alpha(0, p_1^*, \dots, p_1^*)= 0.
	$$
	This verifies the first claim.  The second follows by the observation that the derivative is
	equal to the change in return for players 3-n dropping instead of holding in different scenarios,
	summed against conditional probabilities when one or more of $p_3,\dots, p_n$ are exactly $p_1^*$.
	By symmetry, this will give $(n-2)$ times the change in return when exactly one of $p_3,
	\dots,p_n$ is equal to $p_1^*$ and the others vary freely.  But, since player 2 always holds, in 
	every scenario this difference is equal to $-1$, giving the result.

	To explain a bit further, a player $3\leq j\leq n$ with hand $p_*$ will lose to player 1 if player 1 holds, so in this case player 1 will win an additional $+1$ in virtual ante behond what they would have won otherwise; 
thus, the difference in player 1's winning between dropping and holding of player $j$ is $-1$.
	If player 1 does not hold, then they will gain back $+1$ in virtual ante should player $j$ hold, beyond their return in case player $j$ should drop; thus, again the difference in return is $-1$.
\end{proof}

\begin{corollary}\label{nonbloccor}
	There exist non-bloc strategies for players 2-$n$ forcing a negative return against player 1
	playing with the bloc-optimal strategy $p_1^*=(1/2)^{1/(n-1)}$,
	namely, a combination of the bloc strategy
	$(-p_2^*, \dots, p_n^*)=(p_1^*, p_1^*, \dots p_1^*)$ and a non-bloc
	strategy $(-p_2^*, \dots, p_n^*)=(0, p_1^*+h, \dots p_1^*+h)$ with $h>0$ sufficiently
	small.
\end{corollary}

\begin{proof}
	The optimal bloc strategy forces strictly negative outcome for $p_1^*$ bounded away from $(1/2)^{1/(n-1)}$,
	whereas, by Proposition \ref{nonblocprop}, 
	a non-bloc strategy $(-p_2^*, \dots, p_n^*)=(0, p_1^*+h, \dots p_1^*+h)$ with $h>0$ sufficiently small
	forces a strictly negative outcome for $p_1$ sufficiently near $(1/2)^{1/(n-1)}$, and is bounded elswhere.
	It follows that a convex combination of the two with vanishingly small weight on the non-bloc strategy
	guarantees a strictly negative outcome against any pure strategy $p_1^*$ for player 1.
\end{proof}

\br\label{assymrmk}
One may compute also the positive derivative with respect to $h$ of the winnings of player 2, who holds always,
to see that this is positive of order $n$.  Thus, the remaining $n-2$ players 3-n receive a negative outcome for
small $h$.  
\er

{\bf Conclusion:}
Combining Corollary \ref{nonbloccor} and Remark \ref{nomixrmkn}, we find that \emph{players 2-$n$ can force a strictly 
negative return} $\alpha\leq -c_0<0$ against any strategy, pure or mixed, of player 1.

\subsection{Nonbloc coalition II: winning strategy for players 2-$n$}\label{s:nonblocwin}

Similarly as in the 3-player case, the abstract arguments of the previous subsection may
be replaced by a concrete construction, giving at the same time the stronger result of a
winning strategy for players 2-$n$.

  \begin{lemma}\label{nL1}
	  For some fixed $C_0,C_1,C_2,\delta >0$, and $\epsilon>0$ sufficiently small,
	  the strategy A. $p_2^*=\cdots=p_n^*= (1/2)^{1/(n-1)} -\epsilon$ satisfies
	  $\alpha(p_1,p_2^*, \dots,p_3^*)< C_1 \epsilon2$ for 
	  $p_1\in [(1/2)^{1/(n-1)}- C_0\epsilon, (1/2)^{1/(n-1)} + C_0\epsilon]$,
	  $\alpha(p_1,p_2^*, \dots,p_3^*)<0$ for 
	  $p_1\not \in ((1/2)^{1/(n-1)}- C_0\epsilon,(1/2)^{1/(n-1)} + C_0\epsilon)$,
  and
	  $\alpha(p_1,p_2^*, \dots, p_3^*)\leq - \epsilon/C_2$ for 
	  $p_1\in [(1/2)^{1/(n-1)}- \delta, (1/2)^{1/(n-1)}+ \delta]$.
  \end{lemma}
	
  \begin{proof}
  As in the proof of Lemma \ref{L1}, the estimates follow readily from
  formulae \eqref{formeq}--\eqref{formineq}.
  \end{proof}

  \begin{lemma}\label{nL2}
For some $\delta, \sigma >0$, and $C_3>0$,
	  the strategy B. $(p_2^*,\dots,  p_n^*)= (0, (1/2)^{1/(n-1)} + \delta ,\dots, (1/2)^{1/(n-1)} + \delta )$ 
  satisfies $\alpha(p_1,p_2^*,\dots, p_n^*)<-1/C_3$ 
  for $p_1\in [(1/2)^{1/(n-1)}-\sigma, (1/2)^{1/(n-1)}+\sigma]$. 
  \end{lemma}

 \begin{proof} 
	 As in the proof of Lemma \ref{L2}, it is enough to show that 
  $\alpha( (1/2)^{1/(n-1)},p_2^*,\dots, p_n^*)<0$, the result then following by continuity.
	 Noting, by \eqref{formeq}, that 
	 $$
	 \alpha( (1/2)^{1/(n-1)},0, (1/2)^{1/(n-1)},\dots (1/2)^{1/(n-1)})=0,
	 $$
	 we find that it is enough to show 
		 $(d/dh) \alpha( (1/2)^{1/(n-1)},0, (1/2)^{1/(n-1)}+h,\dots (1/2)^{1/(n-1)}+h)<0$.
		 But, this was already shown in \eqref{keynest}, Proposition \ref{nonblocprop}.
  \end{proof}

Using the results of Lemmas \ref{nL1}-\ref{nL2} and \ref{betalem} 
we may apply, line by line, the proof of Corollary \ref{ABcor} to obtain
the following generalization to the n-player case, $n\geq 4$.

\begin{corollary}\label{nABcor}
	For some fixed $C_0,C_1,C_2,C_3, \delta, \sigma >0$,
	$C>0$ sufficiently large, and $\epsilon=\epsilon(C)>0$ sufficiently small,
  the mixed strategy $ (1-C\epsilon^2)A+ C\epsilon^2 B$ returns $\alpha<0$, $\beta\leq \beta_0<1$ 
	for all $p_1^*\in [0,1]$, and therefore \emph{is a winning strategy for players 2-$n$}.
\end{corollary}

\section{Partial analysis of the discrete $2$-player game}\label{s:discrete}
We now turn to the original, discrete problem, adapting the arguments developed for the continuous case.
Recall that there are different versions of Guts depending on the number of cards $m$ drawn for each 
hand, with $m$ typically equal to $2$, $3$, or $5$.

\subsection{Case $m=1$ (one-card draw) }\label{s:onecard}
Guts could in principle be played with a single-card draw, without much changing the game.
We start with this simplest case $m=1$ to illustrate the approach.

In a standard 52-card deck, cards are strictly lexicographically ordered by number and suit,
with numbers from $2$ to $14$ (counting face cards Jack, Queen, King, Ace as $11$, $12$, $13$, and $14$)
and suits, in increasing order of value, Clubs (C), Hearts (H), Diamonds (D), and Spades (S).
For purposes of comparison they could equivalently be numbered $1$ to $N=52$ by order of their value.
More generally, consider {\it any} strictly ordered deck of $N$ cards, numbered from $1$ to $N$.

We may label the strategies of player 1 and player 2 by $i_1$ and $i_2$, where $i$ means hold for 
any hand with value $>i$.
As each card is equally likely to be selected in a fair draw, the probability of drawing a card
of value $\leq i$ is 
\be\label{p(i)}
p(i):=i/N;
\ee
this relates the discrete strategies $i$ to the probabilities 
$p_j^*$ of the continuous model.
Since cards are dealt without replacement, the conditional probability for the hand of player 2 is affected
by the hand of player 1.
Taking this into account, we define the auxiliary (conditional) probability $\tilde p(i)$ to
be the probability that player 2's card is of value $\leq i$ given that player 1's card is of value $\leq i$,
or
\be\label{tp(i)}
\tilde p(i):=(i-1)/(N-1).
\ee

We define the payoff functions $\alpha(i_1,i_2)$ and $\beta(i_1,i_2)$ similarly as for the continuous case,
as, respectively, the expected one-shot payoff and multiplication of stakes given the choice
of strategies $i_1$ and $i_2$.

\begin{proposition}\label{d2payprop}
For 2-player guts with $m=1$, the payoff function is $\Psi(i_1^*,i_2^*)=\alpha + \beta V$, where 
	\be\label{d2beta}
	\beta(i_1^*, i_2^*)= p(i_1^*)\tilde p(i_2^*)+ (1-p(i_1^*))(1-\tilde p(i_2^*))
\ee
and
	\be\label{d2alpha}
	\alpha(i_1^*,i_2^*)=
	\begin{cases}
		(1-2\tilde p(i_1^*))(p(i_1^*)-p(i_2^*)) & i_2^*\leq i_1^*,\\
		(1-2\tilde p(i_2^*))(p(i_1^*)-p(i_2^*)) & i_2^*> i_1^*.
	\end{cases}
	\ee
\end{proposition}

\begin{proof}
The game terminates unless both players drop, or both players hold, i.e.,
unless $0<i_1<i_1^*$ and $0<i_2<i_2^*$ or $i_1^*\leq i_1\leq 1$ and $i_2^*\leq i_2\leq 1$.
	These are disjoint events with probabilities $p(i_1^*)\tilde p(i_2^*)$ and 
	$(1-p(i_1^*))(1-\tilde p(i_2^*))$. 
	Thus, the probability of replaying the game is 
	$p(i_1^*)\tilde p(i_2^*)+ (1-p(i_1^*))(1-\tilde p(i_2^*))$, and, since the size of the pot
does not change in the $2$-player game, we have therefore immediately
	$\beta(i_1^*, i_2^*)= p(i_1^*)\tilde p(i_2^*)+ (1-p(i_1^*))(1-\tilde p(i_2^*))$ as claimed.

We compute $\alpha$ by the alternative argument of Section \ref{s:2alt}.
Take without loss of generality $i_1^*\leq i_2^*$. By symmetry, $\alpha(i_2^*,i_2^*)=0$.  Writing
$\alpha(i_1^*,i_2^*)$ as the difference $\alpha(i_1^*,i_2^*)- \alpha(i_2^*,i_2^*)$, we may condition
	as in the continuous argument on the case $i_1^*\leq i_1\leq  i_2^*$.
	There are again two subcases: (i) $i_2\geq i_2^*$, in which case player 2 holds and (by our conditioning
	assumption) wins. (ii) $i_2<i_2^*$, in which case player 2 drops and thus loses.
	Meanwhile, the difference in payoff for player 1 between strategy $i_1^*$ and $i_2^*$
	is, by assumption $i_1^*\leq i_2^*$, the difference between player 1 holding and dropping: for
	case (i) (since they lose the whole pot if they hold but only their ante if they drop) $(-2)-(-1)=-1$.
	for case (ii) (since they win the ante if they hold and nothing if they drop) $(+1)-(0)=+1$.

	Computing that case (i) has probability $(p(i_2^*)-p(i_1^*))(1-\tilde p(i_2^*))$ and case (ii) probability
	$(p(i_2^*)-p(i_1^*))\tilde p(i_2^*)$, we thus have an expected difference in return of
	$$
	(p(i_2^*)-p(i_1^*))(1-\tilde p(i_2^*))(-1) + (p(i_2^*)-p(i_1^*))\tilde p(i_2^*)(+1) = 
	(p(i_2^*)-p(i_1^*))(2\tilde p(i_2^*)-1), 
	$$
	as claimed. The formula in case $i_1^*>i_2^*$ then follows by symmetry.
	\end{proof}

\begin{corollary}\label{d2opt}
	For any $N\geq 2$ (the minimum to play a game),
	the largest value $i_1^*$ such that $\tilde p(i_1^*)\leq 1/2$ is an optimal pure strategy for player 1,
	guaranteeing nonnegative return; that is,
	\be\label{d2optimum}
	i_1^*=\begin{cases} N/2 & \hbox{\rm $N$ even},\\
		(N-1)/2, \, (N+1)/2 & \hbox{\rm $N$ odd}.
	\end{cases}
	\ee
\end{corollary}

\begin{proof}
	By \eqref{d2alpha}, $i_1^*$ guarantees a one-shot return $\alpha\geq 0$ if and only if
	(i) $\tilde p(i_1^*)\leq 1/2$, and (ii) $\tilde p(i_2^*)\geq 1/2 $ for any $i_2^*> i_1^*$.
	Taking $i_1^*$ to be the largest value for which (i) is satisfied, we find automatically that (ii) is
	satisfied as well. 
	Moreover, by \eqref{d2beta}, $\beta(i_1^*,\cdot) \leq \theta<1$ unless $p(i_1^*)=1$ or $p(i_1^*)=0$.
	The first cannot happen by definition of $i_1^*$; the second happens only if $1/N>1/2$, or $N=1$, a 
	contradiction.
	Thus, by Theorem \ref{tthm}, strategy $i_1^*$ guarantees a nonnegative return.
	Finally, rewriting conditions $\tilde p(i_1^*)=(i_1^*-1)/(N-1)\leq 1/2$ 
	and $\tilde p(i_1^*+1)=(i_1^*)/(N-1)\geq 1/2$ together as
	$
	(N-1)/2\leq i_*\leq 1+ (N-1)/2, 
	$
	we obtain the result \eqref{d2optimum}.
\end{proof}

\br\label{strictrmk}
Note that we have $\tilde p(i_1^*)< 1/2$ for $N$ even: more precisely,
$$
1- 2\tilde p(i_1^*)= 2 \tilde p(i_1^*+1)-1   = 1/(N-1).
$$
Thus, $i_*$ is a {\it strict saddlepoint}, satisfying the lower bound
\be\label{strictsaddle}
\alpha(i_1^*,i_2^*)\geq |i_2^*-i_1^*|/N(N-1), \quad i_2^*=1, \dots, N.
\ee
That is, over-reckless as well as over-cautious play is penalized in the discrete even case, albeit very slightly,
in contrast to the situation of the continuous case (cf. Rmk. \ref{Notermk}).
For odd $N$, the larger of the two optima satisfies $\tilde p((N+1)/2)= (N-1)/2(N-1)=1/2$,
giving a degenerate saddlepoint, while the smaller satisfies $\tilde p((N-1)/2)= 
%(N-3)/2(N-1)=
1/2- 1/(N-1)<1/2$, giving a strict saddlepoint.
\er

For a standard $52$-card deck, \eqref{d2optimum} gives $ i_*= 26, $
corresponding to the 8 of Hearts (8H).
Thus, for 1-card draw Guts, {\it the optimal strategy is to hold for hands} (strictly) {\it greater than 8H.}

\subsection{Case $m=2$ (two-card draw) }\label{s:twocard}
For a standard 52-card deck, there are 
$N=C(52,m)$ possible hands in $m$-card draw, each equally likely;
in 2-card draw, this is $N=C( 52,2)= 1,326 $.

\smallskip

{\bf Artificial restriction to threshold type.} 
In principle, the possible pure strategies for this game are the set of
subsets, of all $N=C(52,2) \approx 1300$ hands, numbering $2^{C(52,2)}> 2^{1000}$, far too many
for numerical optimization via simplex or other standard methods.
And, unlike the case of $1$-card draw, it is no longer clear that threshold strategies
dominate non-threshold type.
%TODO: but maybe it is true... ??? check this...
For example, a bit of experimentation reveals existence of lower-valued hands that are more likely to win
than a higher-valued hand, due to the fact that the cards of the former (drawn without replacement)
better ``block'' player 2 from receiving a better hand.
However, it seems intuitive that a strategy sufficiently far from threshold type {\it is} dominated by some
threshold type strategy, and that typical threshold-type strategies at least {\it would} dominate non-threshold type.
Thus, as a more tractable starting point, we here
{\it artificially} restrict attention to pure strategies of threshold type, and solve the resulting
game completely, determining a pure, threshold type solution that is optimal against all other threshold type
solutions.

It is our hope that this solution may be shown by further analysis to be optimal against nonthreshold type solutions
as well.
However, we do not carry out such an analysis here, contenting ourselves with the already substantial analysis of the
artificially restricted version of the problem.

\smallskip

{\bf Computations.}
We define $p(i)$ as in the previous case by \eqref{p(i)}.
As for 1-card draw, the probabilities for the hand of player 2 must be conditioned on the hand of player 1 in
the computation of the payoff function. However, the influence of player 1's hand on possible hands for player 2 is more subtle than in case $m=1$, in that it eliminates not only the possibility of the hand itself, but of {\it any hand}
involving one or more of the two cards in player 1's hand, and this depends not only on the value of player 1's hand
but its specific makeup.
Thus, we define a modified auxiliary probability $\tilde p(i:j)$ as 
the probability that player 2's hand is of value $\leq i$ given that player 1's hand is of value $ j$.
For $i_1<i_2$, we define in addition the mean
\be\label{barp}
 \bar p(i_2,i_1): \sum_{i_1< j\leq i_2}\tilde p(i_2:j)/(i_2-i_1)
\ee
of $\tilde p(i_2:j)$ on $i_1<j\leq i_2$.
With these definitions, we obtain the following more general expressions for $\alpha$, $\beta$
valid for any value of $m$, including the previous case $m=1$.

\begin{proposition}\label{d2_2payprop}
For 2-player guts with any $m\geq 1$, the payoff function is $\Psi(i_1^*,i_2^*)=\alpha + \beta V$, where 
	\be\label{d22beta}
	\beta(i_1^*, i_2^*)= p(i_1^*)\bar p(i_2^*,1)+ (1-p(i_1^*))(1-\sum_{j=i_1^*+1}^N \tilde p(i_2^*:j)/(N-i_1^*))
\ee
and
	\be\label{d22alpha}
	\alpha(i_1^*,i_2^*)=
	\begin{cases}
		(1-2\bar p(i_1^*, i_2^*))(p(i_1^*)-p(i_2^*)) & i_2^*\leq i_1^*,\\
		(1-2\bar p(i_2^*, i_1^*))(p(i_1^*)-p(i_2^*)) & i_2^*> i_1^*.
	\end{cases}
	\ee
\end{proposition}

\begin{proof}
	The proofs of \eqref{d22beta}-\eqref{d22alpha} are exactly the same as the ones for 
	\eqref{d2beta}-\eqref{d2alpha} in Proposition \ref{d2payprop}, taking into account additional dependence
	of conditional probabilities in the present case.
\end{proof}

\begin{example}\label{varpeg}
	Despite the apparent similarity of \eqref{d2alpha} to its analogs in the continuous and $m=1$ cases,
	there is an important difference in the dependence of the mean $\bar p(i_1,i_2)$, or, equivalently,
	of the individual conditional probabilitie
	$\tilde p(i_1,i_2)$, on $i_2$.
	We illustrate this with the following example.  Let $i_1$ correspond to the hand 10C/7H, and consider
	$i_2$ corresponding to hands: a) 10C/7C. b) 10C/6S. c) 9S/7C.
	Then, the number of hands of value $\leq i_1$ excluded by player 2 drawing the hand corresponding to i2
	is 59 in case (a), 66 in case (b), and 64 in case (c), as compared to 101 total excluded hands.
	Thus, $\tilde p(i_1,i_2)$ is given by $(i_1-59)/(1326-101)$ in case (a),
	$(i_1-66)/(1326-101)$ in case (b), and
	$(i_1-72)/(1326-101)$ in case (c), all three different values.
	For example, the hand 9S/7C in case (c) excludes all hands including either of the cards 9S or 7C,
	100 involving exactly one of them and 1 involving both, for a total of 101.
	Of these, the ones of value less than or equal to that of 10C/7H are those consisting
	of 9S together with a card of value less than or equal to 8S, including the hand involving both 9S and 7C-
	32 in total-
	plus those consisting of 7C together with a card of value less than or equal to 9S, but excluding the 
	hand 7C/9S already counted and the three pairs 7C/7H, 7C/7D, and 7C/7S- 36-4 =32 in total, for a grand 
	total of 64.
\end{example}

\begin{corollary}\label{d2_2char}
There exists an optimal pure strategy $i_*$ for $m$-card draw guts if and only if there is $i_1^*$ such that
(i) $\bar p(i_1^*,i_2^*)\leq 1/2$ for all $i_2^*< i_1^*$ and
(ii) $\bar p(i_2^*,i_1^*)\geq 1/2$ for all $i_2^*>i_1^*$;
	in particular, 
	\be\label{crit}
	\tilde p(i_1^*,i_1^*-1)\leq 1/2\leq \tilde p(i_1^*+1,i_1^*).
	\ee
%Moreover, if such an optimal strategy exists, it may be characterized as the largest value of $i_1^*$
	%satisfying (i).
\end{corollary}

\subsubsection{Computation of $\bar p(i_1^*,i_2^*$}\label{n:S1comp}
We now specialize to the case $m=2$.
Removing player 1's two cards from the deck eliminates $M:=C(52,2)-C(50,2)= 2 (50) +1=101$ possible hands,
so that 
$$
(i-M)/(N-M)\leq \tilde p(i:j), \, \bar p(i,j)\leq i /(N-M).
$$
Thus, an optimal $i_1^*$ if it exists must satisfy
$ (i_1^*-M)/(N-M)\leq 1/2$ and $(i_1^*+1)/(N-M) \geq 1/2$, giving the crude estimate
$ (N-M-2)/2\leq i_1^*\leq (N+M)/2, $ or $611.5\leq i_1^*\leq 713.5$ a range of $M=101$ possible integer values. 
In particular, this is far from the range where pairs occur, hence we may disregard this possibility in our
calculations from now on. For simplicity, we will assume also $j\geq 2$, throwing out hands on the low end as well.

Denote hand $i_1*$ by $(j_1,k_1)$/$(l_1, m_1)$, where $j, l\in \{1,\dots, 13\}$ denote the numerical value of the card
and $k,m\in \{1,\dots,4\}$ the suit.
We will take without loss of generality $j>l$, assuming that the hand does not involve a pair.
Then, the number $i_1$ of hands of lesser or equal value is
the sum of the number of hands 
$(j_1, k_1)$/$(l_1,m)$ with $m\leq m_1$, or $m_1$;
$(j_1, k)$/$(l_1,m)$ with $k< k_1$, or $4 (k_1-1)$;
$(j_1, k)$/$(l,m)$ with $l< l_1$,  or $16 (l_1-1)$;
and $(j, k)$/$(l,m)$ with $l<j< j_1$, or $16 (j_1-1)(j_1-2)/2$, totaling
\be\label{iform}
i_1= 16\Big( (j_1-1)(j_1-2)/2 +(l_1-1) \Big) +4(k_1-1) + m_1.  
\ee
Thus, the range $612\leq i_1^*\leq 713$ of potential optimal strategies corresponds to hands between
$(10, 1)$/$(3,4)$  and $(10, 2)$/$(9,1)$; 
in particular, top cards of numerical value $10$, or Jack (J).

Similarly, the number of cards of value less than or equal to $i\sim (j,k)$ is
\be\label{Nform}
N(j,k)= 4(j-1)+ k.
\ee
Define now for $i_2<i_1$, $S(i_1,i_2)$ to the the number of player 1 hands of value $\leq i_1$ eliminated by player 2
drawing hand $i_2$, so that
\be\label{Stp}
\tilde p(i_1,i_2)= \frac{i_1-S(i_1,i_2)}{N-M}.
\ee

Then, we have the following sequence of conclusions completing our study.
As the proofs of these results are quite lengthy, we defer them to Appendix
\ref{s:defer} so as not to interrupt the expositional flow.

\bl\label{Slem}
Let $i_r$ correspond to hand $(j_r,k_r)$/$(l_r,m_r)$, with $i_2<i_1$, $m_1\geq 2$, $l_1<j_1$. Then
\be\label{Sform}
S(i_1,i_2)=
\begin{cases}
	%N(j_1,k_1)+ N(j_1-1,4)-7 + k_1 \chi_{ (j_2,k_2)\leq (l_1,m_1)},  &j_1>j_2,\\
	N(j_1,k_1)+ N(j_1-1,4)-9 + k_1 \chi_{ (j_2,k_2)\leq (l_1,m_1)},  &j_1>j_2,\\
	%
	%N(j_1,4) + N(l_1-1,4)- 4 + 4 \chi_{k_2<k_1} + m_1\chi_{k_2=k_1}, & j_1=j_2, \, l_1>l_2,\\
	N(j_1,4) + N(l_1-1,4)- 5 + 4 \chi_{k_2<k_1} + m_1\chi_{k_2=k_1}, & j_1=j_2, \, l_1>l_2,\\
	%
	%N(j_1,k_1-1) + N(l_1,4)- 4 +  \chi_{m_2\leq m_1}, & j_1=j_2, \, l_1=l_2, \,  k_1>k_2,\\
	N(j_1,k_1-1) + N(l_1,4)- 5 +  \chi_{m_2\leq m_1}, & j_1=j_2, \, l_1=l_2, \,  k_1>k_2,\\
	%
	%N(j_1,k_1) + N(l_1,m_1), & j_1=j_2, \, l_1=l_2, \, k_1=k_2. 
	N(j_1,k_1) + N(l_1,m_1)-5, & j_1=j_2, \, l_1=l_2, \, k_1=k_2. 
	%REDUNDANT \, m_1>m_2.
\end{cases}
\ee
\el

\bc\label{1crit}
The criterion \eqref{crit} is satisfied \emph{only for} $i_1^*=669\sim (j_1^*,k_1^*)/(l_1^*,m_1^*)=(10,4)/(
,1)$.
\ec

\emph{Corollary \ref{1crit} narrows our search for an optimal pure solution down to the single hand JS/7C.}

Having narrowed the search to one candidate $i_*\sim (10,4)/(6,1)$, we now
check the averaged condition of Corollary \ref{d2_2char} for 
the $[101/2]= 50$ hands $i_2$ above and below $i_*$
%, $100$ hands in all, 
to verify optimality.

\begin{corollary}\label{d2optthm}
	For 2-player 2-card Guts, the pure strategy $i_1^*\sim (10,4)/(6,1)$, corresponding to (holding for hand
	of value greater than) JS/7C,
	is optimal, guaranteeing a nonnegative return.
\end{corollary}

\br\label{strictm2rmk}
Similarly as for the 1-card draw case, we have strict inequality $\bar p(i_1^*, i_2)<1/2$ for $i_2<i_1^*$ and
$\bar p(i_2, i_1^*)>1/2$ for $i_2>i_1^*$, so that the symmetric equilibrium $i_j^*=669$ is of \emph{strict}
type, penalizing deviation from equilibrium by either player.
On the other hand, the crude estimate 
$$
|\bar p(i_1, i_2)- p(i_2)|\leq \Big| \frac{i_1\pm M}{N-M}-\frac{i_1}{N}\Big|
=\frac{2M}{N-M}
=2 \frac{101}{1225} \approx .083
$$
shows that the penalty is rather small, of order $(.08)|p(i_2)-p(i_1^*)|$.
\er

{\bf Note:} We emphasize that this result is for the modified game with
pure strategies restricted to be of threshold type.  A complete analysis must verify optimality also
against nonthreshold type.

%\subsection{Cases $m=3$ or $5$ (3-card Monte or 5-card draw) }\label{s:manycard}
%We conjecture that similar but more complicated analyses may be carried out for the popular variants $m=3,5$.
%with a similar result: namely, that there is an optimal mixed strategy giving nonnegative return,
%%but {\it no pure ``go/no-go'' strategy} that guarantees a nonnegative, or ``break-even'', result.

\section{Discussion and open problems}\label{s:disc}
In summary, we have provided an analytic framework for noncontractive generalized recursive games and used it to
treat the interesting practical example of Guts Poker.
For a simplified continuous model, we have treated the general n-player game, at least in its broad outlines.
Our main result is that there exists a symmetric Nash equilibrium, but that this equilibrium is nonstrict
for $n\geq 2$ and for $n\geq 3$ is nonstrong. 
This equilibrium consists of strategies $p_j^*=(1/2)^{1/(n-1)}$, 
for which each player ``holds'' precisely if their cards are of a value that exceeds a randomly chosen hand with
probability $\geq p_j^*$.
It is nonstrict in that a single player j may deviate toward more ``reckless'' play, decreasing $p_j^*$ with
no penalty in expected return.
It is nonstrong for $n\geq 2$ in that a coalition of $n-1$ players may deviate from equilibrium in a way that
guarantees them an improved joint return.

More, the symmetric equilibrium strategy $p_1^*=(1/2)^{1/(n-1)}$ is an optimal pure strategy
guaranteeing nonnegative return
for player 1 against ``bloc'' strategies in which players 2-$n$ behave identically, 
not just on the average but round by round.  However, for $n\geq 3$, 
if players 2-$n$ are allowed to play in arbitrary fashion as
a coalition, then they can force a strictly positive joint return for themselves, hence a strictly negative
return for player 1. Indeed, we have provided a simple and explicit example of such a ``winning'' coalition strategy,
in which with small probability a designated player holds always while the other coalition players hold slightly less
often to compensate. The aggressive (hold always) player will have an overall positive return outweighing the 
slightly negative return of their colleagues.

An interesting followup would be to approximate numerically an optimal strategy for players 2-$n$, that is, to
determine the value for the coalition-based game. In particular, it would be interesting to know how close our
simple winning strategy is to being optimal.
Another question of possible interest is whether there is a joint strategy for players 2-$n$ in which not only the
joint return, but the return of each separate player is positive.
Of course, by randomly alternating the roles of the coalition players, this can be achieved on average; however,
the question we are getting at is whether there is a type of coalition strategy from which no player among
2-$n$ would be tempted to depart. In particular, could there be a different {\it mixed-type} Nash equilibrium
that is symmetric only among players 2-$n$, and more favorable to them, returning a positive payoff to each?

For the 2-player game we have gone a bit farther, determining the optimal strategy for the original, discrete
game for Guts Poker with 1-card draw and, under an artificial restriction of pure strategies to threshold type,
with 2-card draw.
For both the continuous and discrete games, the optimal strategy is of ``pure'' type, i.e., a simple ``go/no-go''
strategy in which the player holds if and only if their hand is above a certain value.
From the game-theoretic point of view, this is quite special, with typical optimal strategies 
being ``mixed'', or random \cite{vN,O}.
From the poker point of view, it is perhaps intuitively appealing, at least from the optimistic point of view
of, e.g., \cite{S}.
However, judging from our results for the continuous case, it is almost certainly not the case for n-player
Guts, $n\geq 3$.
This would be very interesting to confirm by a discrete analysis of the $3$ or higher $n$ case similar to
that carried out here for the 2-player game.
Likewise, a very interesting open problem is to complete the analysis of the discrete 
2-player game for 2-card draw, by expanding the
analysis to general pure strategies. Extension of the discrete analysis to the popular alternatives
of $3$- and $5$-card guts would be very interesting as well.

Interestingly, the discrete 2-player equilibrium is {\it strict} for 1- or 2-card draw, 
due to Diophantine considerations, unlike the continuous model it approximates. 
Thus, it penalizes any deviation of players from the equilibrium strategy.
However, the size of the penalty is rather small, as may be seen by closeness of continuous and discrete
payoff functions, Remark \ref{strictm2rmk}.
Thus, indeed, the continuous model seems to be a useful organizing center for analysis of the full, discrete game.

In terms of the abstract theory of generalized recursive games, a very interesting direction for further
study would be extension of the results for single-state games in Section \ref{s:value} to the general case 
of multi-state generalized recursive games.

\appendix
\section{Nash equilibria vs. von Neumann-Morgenstern coalitions}\label{s:nashvN}
It seems interesting to connect more closely the questions investigated here, of player 1's outcomes
against bloc and nonbloc strategies of players 2-$n$, to those of the standard literature.
The point of view taken here in the nonbloc case is similar to those of von Neumann-Morgenstern \cite{vNM}
and Cournot \cite{C}, in which an $n$-player game is viewed as a 2-player game between different coalitions,
while our bloc analysis is somewhat reminiscent of that of Nash \cite{N}.
The following results give precise connections between these various ideas for (one-shot)
symmetric finite zero-sum matrix games, informing (in hindsight) our study of the more general continuous 
generalized recursive (so not necessarily zero-sum) case.

\begin{proposition} [Strong Nash equlibrium vs. 2-$n$ coalition]
For symmetric finite zero-sum games, 
there exists a strong symmetric Nash equilibrium with strategy $s$ if and only if
player 1 can force return $\geq 0$ using $s$ vs. a coalition of players 2-$n$.
\end{proposition}

\begin{proof}
(only if)
At a symmetric equilibrium $s$, the return to all players is $0$.
	If some subcoalition of players 2-$n$ can improve their joint return by deviating from
the symmetric equilibrium, then each nondeviating player separately loses, by symmetry, 
	in particular player 1. 
Thus, the full coalition of all players 2-$n$ can improve their joint return by deviating from 
	the symmetric equilibrium. In other words, a symmetric equilibrium  $s$ is strong if and only
	if it penalizes deviations by the specific coalition 2-$n$, i.e.,
	player 1 can force with strategy $s$ a return $\geq \alpha(s, \dots, s)=0$.

	(if) If there exists a player 1 strategy $s$ guaranteeing $\geq 0$ return,
	then by symmetry, $s$ is a symmetric Nash equilibrium, and, by the argument of the previous case
	it is strong.
\end{proof}

\begin{proposition} [Value for symmetric bloc case]
For symmetric finite zero-sum games, the value of the 2-player game pitting player 1 vs. bloc strategies
of players 2-$n$ is less than or equal to $0$.
	If the optimal (mixed) bloc strategy is pure (deterministic), or more generally of
	form $(s,\dots, s)$ with $s$ an individual mixed strategy, then the value is equal to $0$.
\end{proposition}

\begin{proof}
	By Nash' Theorem \cite{N}, there exists a symmetric Nash equilibrium $(s,\dots, s)$, whose value
	by symmetry is $0$.
	By the definition of Nash equilibrium, the bloc strategy choice $(s,\dots, s)$ for players 
	2-$n$ gives payoff to player 1 less than or equal to the value when player 1 chooses strategy $s$,
	which is $0$. This proves the first assertion.
	By the Fundamental Theorem of Games, the value $V$ is equal to the maximum return
	forceable by player 1 and also the minimum return forceable by players 2-$n$.
	If the optimal bloc strategy $s$ is of form $(s,\dots,s)$, then player 1 can play $s$ against it
	to achieve value $0$. Hence, $V\geq 0$, giving $V=0$ when combined with the first
	observation. This proves the second assertion.
	%NO! not a classical game! for, that would be combinations of bloc strategies, sort of like asynch.
%If $s$ is optimal for player 1, then players 2-$n$ can play $(s,\dots,s)$ to achieve payoff $0$.
%Likewise, if $(s,\dots,s)$ is optimal for 2-$n$, player 1 can play $s$ to achieve payoff $0$.
%Thus, $V\geq 0$ and $V\leq 0$, and so $V=0$.
\end{proof}

\br\label{finepointrmk}
Note that mixed strategies for the bloc game consist of combinations of pure bloc strategies,
which are not necessarily of form $(s,\dots,s)$. Thus, in general, the value of the bloc game
could be strictly less than $0$.
\er

%HERE

\begin{proposition}[Symmetric Nash equilibrium vs. optimum bloc strategy]
	Consider symmetric finite zero-sum games with symmetric Nash equilibrium $s$.
	If the Nash equilibium is strong then it is an 
optimum in the bloc strategy game for both player 1 and the coalition of players 2-$n$.  
	Likewise, if optimum bloc strategies are strict optima of form $(t,\dots,t)$--
	in particular, if they are pure (deterministic) strategies--
	then they are strict symmetric Nash equilibria, i.e., $t=s$.
\end{proposition}

\begin{proof}
	The first assertion follows by the definition of strong Nash equilibrium.
Likewise, if $(s, t,\dots,t)$ is an optimal (mixed) saddlept., 
then, by the previous proposition, $0=\alpha(s,t,\dots,t)=\alpha(t,\dots,t)$, hence $s=t$ by strictness. 
\end{proof}

\section{The ``Weenie rule''}\label{s:weenie}
An interesting variant on Guts Poker is the addition of the ``Weenie rule'' \cite{S}
penalizing overcautious play.
Under this rule, should all players drop, the player with highest hand- the ``weenie''- must match the
pot, thus doubling the stakes for the next round.  
In this appendix, we test our analytic framework/approach by analyzing this modification and
its effect on optimal strategy.

\subsection{Bloc strategy case}\label{s:wbloc}

\begin{proposition}[Modified payoff]\label{weenieprop}
	With the Weenie rule, we have for $p_1^*\leq p_2^*$
	\be \label{wform1}
	\alpha(p_1^*, p_2^*, \dots, p_2^*)= (n-1)(p_2^*-p_1^*)\Big(2(p_2^*)^{n-1}-1\Big)+
\Big((p_2^*)^{n-1} - (p_1^*)^{n-1} \Big)p_1^*,
\ee
	while for $p_1^*> p_2^*$,
	with $\delta$ $=0$ for $n=2$ and $>0$ for $n\geq 3$ satisfying \eqref{corrector},
	\be \label{wform2}
	(n-1)(p_2^*-p_1^*)\Big( 2(p_1^*)^{n-1}+ (p_2^*)^{n-1}-1 \Big)+ 
	\delta =
	\alpha(p_1^*, p_2^*, \dots, p_2^*)< (n-1)(p_2^*-p_1^*)\Big(3(p_2^*)^{n-1}-1\Big).
	\ee
\end{proposition}

\begin{proof}
%TODO: smooth
	We compute the contribution to expected return $\alpha$ due to the Weenie rule, to
	be added to the value computed for standard Guts in Proposition \ref{blocprop}.

	({\it Case $p_1^*\leq p_2^*$}) The relevant scenario is that all players drop,
	which occurs with probability $p_1^*(p_2^*)^{n-1}$.
	Subcases: 
	(i) $p_2,\dots,p_n\leq p_1^*$, probability $(p_1^*)^{n-1}p_1^*$, expected return zero (symmetric
among all players), and (ii) at least one of $p_j^*\geq p_1^*$, probability $((p_2^*)^{n-1}- (p_1^*)^{n-1})p_1^*$, 
expected return $+1$ (additional ante).  Summing probabilities times returns, we obtain
	a total expected contribution of
\be\label{total1}
((p_2^*)^{n-1} - (p_1^*)^{n-1})p_1^*
\ee
to be added to the value of $ \alpha$ previously computed in Proposition \ref{blocprop}.
%NOTES, check: Derivative at $p_1^*=p_2^*$ evidently $-(n-1)(p_1^*)^{n-1}$

	({\it Case $p_1^*\geq p_2^*$}) 
	Again, the relevant scenario is that all players drop, occurring with 
	probability  $p_1^*(p_2^*)^{n-1}$.
Subcases: 
	(i) $p_1\leq p_2^*$, probability $(p_2^*)^{n-1}$, expected return zero (symmetric among all players),
and (ii) $p_2^*\leq p_1\leq p_1^*$, probability $(p_2^*)^{n-1}(p_1^*- p_2^*)$,
expected return $-(n-1)$ ($-n$ Weenie penalty, $+1$ virtual ante).  Summing, we obtain a total expected
contribution
\be\label{total2}
(n-1)(p_2^*)^{n-1}(p_2^* - p_1^*)
\ee
to be added to previously computed $ \alpha$. 
%NOTES, check: Derivative $-(n-1)(p_2^*)^{n-1}$, same as previous case.

	Summing \eqref{total1} and \eqref{total2} with \eqref{formeq} and \eqref{formineq}, we obtain
	%\be\label{formeq} \be\label{formineq}
	\eqref{wform1} and \eqref{wform2}.
\end{proof}

\begin{corollary}\label{wblocprop2}
	For $n$-player Guts with the Weenie rule, the strategy $p_1^*=(1/3)^{1/(n-1)}$ 
	is optimal with respect to bloc strategies $p_2^*=p_3^*=\dots=p_n^*$,
	forcing expected return $\geq 0$.
	Likewise, the bloc strategy $p_j^*=(1/3)^{1/(n-1)}$, $j=2,\dots, n$ is optimal for players 2-$n$,
	forcing joint return $\geq 0$.
	Moreover, both 
	%optimal strategies 
	are strict optima; in particular, $p_j^*\equiv (1/3)^{1/(n-1)}$ is
	a \emph{strict Nash equilibrium}.
\end{corollary}

\begin{proof}
	(Essentially identical to that of Corollary \ref{blocprop2}.)
	Taking $p_1^*=(1/3)^{1/(n-1)}$ in Proposition \ref{weenieprop}, we find for $p_2^*\geq (1/3)^{1/(n-1)}$
	by \eqref{wform1}, that
	$$
	\alpha((1/3)^{1/(n-1)}, p_2^*, \dots, p_2^*) \geq (n-1)(p_2^*-(1/3)^{1/(n-1)})
	\Big(3(p_2^*)^{n-1}-1\Big)\geq 0,
	$$
	with strict inequality for $p_2^*> (1/3)^{1/(n-1)}$.
	Taking $p_1^*=(1/3)^{1/(n-1)}$ in Proposition \ref{weenieprop}, we find for $p_2^*\geq (1/3)^{1/(n-1)}$
	For $p_2^*< (1/3)^{1/(n-1)}$, the lefthand inequality of \eqref{wform2} gives
	$$
	\alpha((1/3)^{1/(n-1)}, p_2^*, \dots, p_2^*) > (n-1)(p_2^*-(1/3)^{1/(n-1)})\Big((p_2^*)^{n-1}-1/3\Big)> 0.
	$$
	Applying \eqref{betalem} and Theorem \ref{tthm}, we obtain the first assertion.

	Similarly, taking $p_2^*=(1/3)^{1/(n-1)}$ gives for $p_1\leq (1/3)^{1/(n-1)}$ by the Mean Value Theorem
	$$
	\begin{aligned}
	\alpha(p_1^*, (1/3)^{1/(n-1)},\dots, (1/3)^{1/(n-1)})
		&=
	(n-1)((1/3)^{1/(n-1)}-p_1^*)\Big(2(p_2^*)^{n-1}-1\Big)\\
		&\quad + \Big((p_2^*)^{n-1} - (p_1^*)^{n-1} \Big)p_1^*\\
		&
	\leq (n-1)((1/3)^{1/(n-1)}-p_1^*)\Big(3(p_2^*)^{n-1}-1\Big)= 0,
	\end{aligned}
	$$
	with strict inequality for $p_1< (1/3)^{1/(n-1)}$. 
	For $(1/3)^{1/(n-1)}< p_1^*$, the righthand inequality of \eqref{wform2} gives
	$
	\alpha(p_1^*, (1/3)^{1/(n-1)},\dots, (1/3)^{1/(n-1)}) < 0. 
	$
Both optimal strategies are strict optima as we have noted along the way.
	But, strict optimality of $p_j^*=(1/3)^{1/(n-1)}$, $j=2,\dots, n$ implies by definition that
	$(p_1^*, \dots, p_n^*)= ((1/3)^{1/(n-1)}, \dots, (1/3)^{1/(n-1)})$ is a strict Nash equilibrium.
\end{proof}

\subsection{General (nonbloc) case}\label{s:wnonbloc}
The fact that the Nash equilibrium with Weenie rule is \emph{strict}, different from the standard case,
makes analysis of the nonbloc case somewhat more straightforward.
For, either the player 1 strategy $p_1^*=(1/3)^{1/(n-1)}$ is optimal against general strategies of
players 2-$n$, in which case the Nash equilibrium is \emph{strong}, 
or else one may readily construct a winning strategy for players 2-$n$ consisting of a combination of
the equilibrium strategy with a vanishingly small multiple of any pure strategy $(p_2^*, \dots, p_n^*$
strictly penalizing $p_1^*=(1/3)^{1/(n-1)}$.
For, it is easily seen that this guarantees $\alpha<0$, while, by continuity, satisfaction of 
$\beta$-criterion \eqref{tcond} is inherited from the bloc case.

		{\bf Conclusion:} The Nash equilibrium $(1/3)^{1/(n-1)}$ with Weenie rule in
		effect is substantially shifted from the equilibrium $(1/2)^{1/(n-1)}$ without.
		Moreover, different from the standard case, this equilibrium is \emph{strict}.
		It is an interesting open question whether or not it is {\it strong} for $n\geq 3$.

\section{Deferred proofs for discrete 2-player case}\label{s:defer}

\begin{proof}[Proof of Lemma \ref{Slem}]
	Hands eliminated are those involving $(j_2,k_2)$, $(l_2,m_2)$, or both.
	Of those eliminate, we must count those of value less than or equal to $(j_1,k_1)$/$(l_1,m_1)$.
	There are 4 cases:

	\medskip

	{\it Case 1.} ($j_1>j_2$)
	Here, the eliminated hands $\leq (j_1,k_1)/(l_1,m_1)$ include those pairing 
	$(l_2,m_2)$ with cards of value $\leq (j_1,k_1)$ 
	together with those pairing $(j_2,k_2)$ with cards of value $\leq (j_1-1,4)$,
	minus 1 for double count of $(j_2,k_2)/(l_2,m_2)$ itself, minus 8 for hands that are pairs or duplicates.
	If $(j_2,k_2)\leq (l_1,m_1)$, then there are an additional $k_1$ pairings of $(j_2,k_2)$ with cards $(j_1,k)$
	with $k\leq k_1$. Total:
	$$
	N(j_1,k_1)+ N(j_1-1,4)-7 + k_1 \chi_{ (j_2,k_2)\leq (l_1,m_1)}.
	$$

	We note further that $(j_2,k_2)\leq (l_1,m_1)$ if $j_2< l_1$ or $j_2=l_1$ and $k_2\leq m_1$.

	\medskip

	{\it Case 2.} ($j_1=j_2$, $l_1>l_2$)
	Here, the eliminated hands $\leq (j_1,k_1)/(l_1,m_1)$ include those pairing 
	$(j_2,k_2)$ with cards of value $\leq (l_1-1,4)$, together with those pairing 
	$(l_2,m_2)$ with cards of value $\leq (j_1,4)$ 
	minus 1 for double count and 4 for possible pair or duplicate of $(l_2,m_2)$.
	If $k_2< k_1$, there are an additional 4 pairings of $(j_2,k_2)$ with
	cards $ (l_1, m)$, and if $k_2=k_1$, an additional $m_1$ pairings with cards $(l_1,m)$ with $m\leq m_1$.
	Total:
	$$
	N(j_1,4) + N(l_1-1,4)- 5 + 4 \chi_{k_2<k_1} + m_1\chi_{k_2=k_1}.
	$$

	\medskip

	{\it Case 3.} ($j_1=j_2$, $l_1=l_2$, $k_1>k_2$)
	Here, the eliminated hands $\leq (j_1,k_1)/(l_1,m_1)$ include those pairing 
	$(l_2,m_2)$ with cards of value $\leq (j_1,k_1-1)$ 
	together with those pairing $(j_2,k_2)$ with cards of value $\leq (l_1,4)$,
	minus 1 for double count and 4 for possible pairs or duplicates.
	If $m_2\leq m_1$, there is an additional 1 pairing of
	$(l_2,m_2)$ with card $(j_1,k_1)$. 
	Total:
	$$
	N(j_1,k_1-1) + N(l_1,4)- 5 +  \chi_{m_2\leq m_1}.
	$$

	\medskip

	{\it Case 4.} ($j_1=j_2$, $l_1=l_2$, $k_1=k_2$, $m_1>m_2$ (redundant))
	Here, the eliminated hands $\leq (j_1,k_1)/(l_1,m_1)$ include those pairing 
	$(l_2,m_2)$ with cards of value $\leq (j_1,k_1)$ 
	together with those pairing $(j_2,k_2)$ with cards of value $\leq (l_1,m_1)$,
	minus 1 for double count and 4 for possible pairs or duplicates.
	Total:
	$$
	N(j_1,k_1) + N(l_1,m_1)-5.
	$$
\end{proof}

\begin{proof}[Proof of Corollary \ref{1crit}]
	{\it Case 1.} ($2\leq m\leq 3$)
	In this case, $i_1-1= (j,k)/(l,m-1)$, $i_1= (j,k)/(l,m+1)$, and $i_1+1= (j,k)/(l,m+1)$, hence
	by \eqref{Sform}(iv)
	$$
	S(i,i-1)=N(j,k) + N(l,m)-5= 4(j-1)+ k + 4(l-1)+ m -5
	$$
	and
	$$
	S(i+1,i)=N(j,k) + N(l,m+1)-5= 4(j-1)+ k + 4(l-1)+ m+-5,
	$$
	and, by \eqref{Stp}, $\tilde p(i,i-1)= \frac{i-S(i,i-1)}{N-M}= \tilde p(i+1,i)$.
	By \eqref{crit}, therefore, $ 1/2=\tilde p(i,i-1)= \frac{i-S(i,i-1)}{N-M}$, 
	or
	$ i=	\frac{N-M}{2} + 4(j+l-3)-1  + k + m$.
	Since $(N-M)/2=612.5$ is not an integer, this gives a contradiction.

	\medskip
	{\it Case 2.} ($m=1$)
	In this case, $i-1= (j,k-1)/(l,1)$ if $k>1$ and
	$i-1= (j,k)/(l-1,4)$ if $k=1$, while $i+1=(j,k)/(l,2)$.
	In either situation, by \eqref{Sform}(iv)
	\be\label{Sform2}
	S(i+1,i)= 
	N(j,k) + N(l,m+1)-5= 
	%4(j-1)+ k+  4(l-1) + 2 -5=
	4(j-1)+  4(l-1) +k- 3. 
	\ee

	({Subcase \it $k>1$}).  Here, $i=(j,k)/(l,1)$, $i-1= (j,k-1)/(l,4)$, and so
	by \eqref{Sform}(iii)
	$$
	S(i,i-1)=
	N(j,k-1) + N(l,4)-5 + \chi_{3\leq 1} = 
	%4(j-1) + k-1 + 4(l-1) + 4 -5 + 0
	%=4(j-1) + k + 4(l-1)-2   CORRECT.
	4(j-1)+ 4(l-1) +k-2. 
	$$
%N(j,k)= 4(j-1)+ k.

	This gives in particular $S(i+1,i)= S(i,i-1) - 1$,
	hence, by \eqref{Stp}, 
	$$
	\begin{aligned}
		\tilde p(i+1,i)&= \frac{i+1-S(i+1,i)}{N-M}\\
		&= \frac{i-S(i,i-1)+2}{N-M}
		= \tilde p(i,i-1) + \frac{2}{N-M}.
	\end{aligned}
	$$

By \eqref{Stp}, therefore
	$$
	\begin{aligned}
		\tilde p(i+1,i)&= \frac{i+1-S(i+1,i)}{N-M}\\
		&= \frac{i-S(i,i-1)+2}{N-M}
		= \tilde p(i,i-1) + \frac{ 2}{N-M}.
	\end{aligned}
	$$

		Meanwhile,
		$$
	\begin{aligned}
		1/2\geq \tilde p(i,i-1)&= \frac{i-S(i,i-1)}{N-M}\\
		&=
		 \frac{
			 16\Big( (j-1)(j-2)/2 +(l-1) \Big)+ 4(k-1)  + 1  
		-\Big(4(j-1)+  4 (l-1)+k-2 \Big)
	%4(j-1)+ 4(l-1) +k-1. 
		 }{N-M}\\
		 %CHECK, so far so good.  673-61= 112, ok.
		&=
		 \frac{
			 16\Big( (j-1)(j-2)/2 +(l-1) \Big) + 3(k-1)  -4\Big(j+  l -2 \Big)
		 }{N-M}\\
	\end{aligned}
		 $$
		 together with the crude bound $j=10$, yields
$$
16\Big( (j-1)(j-2)/2 +(l-1) \Big)  -4\big(j+  l -2 \Big) + 3(k-1)+2 \leq \frac{N-M}{2}=612.5,
$$
or
$$
12 l + 3(k-1)   \leq 82.5,
$$
while  $\tilde p(i+1,i)\geq 1/2$ on the other hand gives
$$
12 l + 3k-1   \geq 82.5,
$$
the unique integer solution of which is $l=6$ $k=4$.
With $j=10$ and $m=1$, this gives $(10,4)/(6,1)$.

	({Subcase \it $k=1$}).  
	Here, $i=(j,k)/(l,1)$, $(i-1)\sim (j,k)/(l-1,4)$, and so
	by \eqref{Sform}(ii)
	\be\label{Sform1}
	S(i,i-1)=
	N(j,4) + N(l-1,4)- 5 + 1  = 4(j-1)+  4 (l-1) ,
	%4(j-1)+4 + 4(l-2)+4 -5 + 1=
	%4(j-1) + 4(l-1)   =
	\ee
%N(j,k)= 4(j-1)+ k. CHECK. works now
giving $S(i+1,i)= S(i,i-1) -2$,

By \eqref{Stp}, therefore
	$$
	\begin{aligned}
		\tilde p(i+1,i)&= \frac{i+1-S(i+1,i)}{N-M}\\
		&= \frac{i-S(i,i-1)+2}{N-M}
		= \tilde p(i,i-1) + \frac{ 2}{N-M}.
	\end{aligned}
	$$

		Meanwhile,
		$$
	\begin{aligned}
		1/2\geq \tilde p(i,i-1)&= \frac{i-S(i,i-1)}{N-M}\\
		&=
		 \frac{
			 16\Big( (j-1)(j-2)/2 +(l-1) \Big)+ 4(k-1)  + 1  
		-\Big(4(j-1)+  4 (l-1)+k-2 \Big)
	%4(j-1)+ 4(l-1) +k-1. 
		 }{N-M}\\
		 %CHECK, so far so good.  673-61= 112, ok.
		&=
		 \frac{
			 16\Big( (j-1)(j-2)/2 +(l-1) \Big) + 3(k-1)  -4\Big(j+  l -2 \Big)+ 1
		 }{N-M}\\
	\end{aligned}
		 $$
%i_1= 16\Big( (j_1-1)(j_1-2)/2 +(l_1-1) \Big) +4(k_1-1) + m_1.  
		 %S(i,i-1)= 4(j-1)+ 4(l-1) +k. 
		 %CHECK, still correct...
		 together with the crude bound $j=10$, yields
$$
16\Big( (j-1)(j-2)/2 +(l-1) \Big)  -4\big(j+  l -2 \Big) + 3(k-1)+1 \leq \frac{N-M}{2}=612.5,
$$
%CHECK, still correct to here.
or
$$
12 l + 3(k-1)   \leq 83.5,
$$
%giving $l\leq 7$ for $k=1$ and $l\leq 6 $ for $k>1$.
while  $\tilde p(i+1,i)\geq 1/2$ on the other hand gives
$$
12 l + 3(k-1) + 2   \geq 83.5,
$$
which has no integer solution.

\medskip

Thus, for $m=1$ and any $k$, we have \eqref{Sform2}, \eqref{Sform1}.
By \eqref{Stp}, therefore
	$$
	\begin{aligned}
		\tilde p(i+1,i)&= \frac{i+1-S(i+1,i)}{N-M}\\
		&= \frac{i-S(i,i-1)+3}{N-M}
		= \tilde p(i,i-1) + \frac{ 3}{N-M}.
	\end{aligned}
	$$

		Meanwhile,
		$$
	\begin{aligned}
		1/2\geq \tilde p(i,i-1)&= \frac{i-S(i,i-1)}{N-M}\\
		&=
		 \frac{
			 16\Big( (j-1)(j-2)/2 +(l-1) \Big)+ 4(k-1)  + 1  
		-\Big(4(j-1)+  4 (l-1)+k-1 \Big)
	%4(j-1)+ 4(l-1) +k-1. 
		 }{N-M}\\
		 %CHECK, so far so good.  673-61= 112, ok.
		&=
		 \frac{
			 16\Big( (j-1)(j-2)/2 +(l-1) \Big) + 3(k-1)  -4\Big(j+  l -2 \Big)+1
		 }{N-M}\\
	\end{aligned}
		 $$
%i_1= 16\Big( (j_1-1)(j_1-2)/2 +(l_1-1) \Big) +4(k_1-1) + m_1.  
		 %S(i,i-1)= 4(j-1)+ 4(l-1) +k. 
		 %CHECK, still correct...
		 together with the crude bound $j=10$, yields
$$
16\Big( (j-1)(j-2)/2 +(l-1) \Big)  -4\big(j+  l -2 \Big) + 3(k-1)+1 \leq \frac{N-M}{2}=612.5,
$$
%CHECK, still correct to here.
or
$$
12 l + 3(k-1)   \leq 83.5,
$$
%giving $l\leq 7$ for $k=1$ and $l\leq 6 $ for $k>1$.
while  $\tilde p(i+1,i)\geq 1/2$ on the other hand gives
$$
12 l + 3k   \geq 83.5,
$$
the unique integer solution of which is $l=6$ $k=4$.
With $j=10$ and $m=1$, this gives $(10,4)/(6,1)$.

	\medskip
	{\it Case 3.} ($m=4$)
	In this last case, $i-1 \sim (j,k)/(l,m-1)$ and so
	$$
	%S(i,i-1)=N(j,k) + N(l,m)-4= 4(j-1)+  4(l-1)+k+ m -5
	S(i,i-1)=N(j,k) + N(l,m)-5= 4(j-1)+  4(l-1)+ k-1
	$$
	as in case 1, but now $(i+1)\sim (j,k)/(l+1,1)$ if $k=4$ and
	$i+1\sim (j,k+1)/(l,4)$ if $k<4$.

	\medskip

	({\it Subcase $k=4$}).
	In this case, by \eqref{Sform}(ii), we have
	%N(j_1,4) + N(l_1-1,4)- 5 + 4 \chi_{k_2<k_1} + m_1\chi_{k_2=k_1}, & j_1=j_2, \, l_1>l_2,\\
	$$
	S(i+1,i)= N(j,4) + N(l,4)- 4 
	%= 4(j-1)+ k + 4 (l-1) + 4 -4
	= 4(j-1) + 4 (l-1)+k.
	$$
	%NOTES: is it??? be careful now!
	%S(i+1,i)= = 4(j-1) + 4 (l-1)+k
	%S(i,i-1)=N(j,k) + N(l,m)-5= 4(j-1)+  4(l-1)+ k-1
	Thus, $S(i+1,i)=S(i,i-1)+1$, and so $\tilde p(i,i-1)=\tilde p(i,k-1)$, leading to a contradiction as in Case 1.

	\medskip

	({\it Subcase $k<4$}).
	In this case, by \eqref{Sform}(iii), we have
	%NOTES: (iii)
	%(j,k+1)/(l,4),  m_1=m_2=4, 
	%N(j_1,k_1-1) + N(l_1,4)- 5 +  \chi_{m_2\leq m_1}, & j_1=j_2, \, l_1=l_2, \,  k_1>k_2,\\
	%N(j,k) + N(l,4)- 5 + 1
	$$
	S(i+1,i)= N(j,k) + N(l,4)- 4
	%= 4(j-1)+ k + 4 (l-1) + 4 -5
	%= 4(j-1) +k+ 4 (l-1)+ 4-4
	= 4(j-1) + 4 (l-1)+ k.
	$$
	%S(i+1,i)= 4(j-1) + 4 (l-1)+ k.
	%S(i,i-1)=N(j,k) + N(l,m)-5= 4(j-1)+  4(l-1)+ k-1
	Thus, $S(i+1,i)=S(i,i-1)+1$, and so $\tilde p(i,i-1)= \tilde p(i_1,i)=1/2$, 
	leading to a contradiction as in Case 1.

 \medskip

 Summarizing, $(j,k)/(l,m)=(10,4)/6,1)$ is the unique possible pure optimal strategy.
\end{proof}

\begin{proof}[Proof of Corollary \ref{d2optthm}]
	As we need only  check $50$ hands $i\sim (j,k)/(l,m)$ above and below $i_*$, we may as discussed
	earlier take without loss of generality $j=10$.

	\medskip

	({\it Case 1. $i<i_*$.}) We consider each of cases (i)-(iv) of \eqref{Sform} in turn.
	Case i ($10>j$) does not occur so may be ignored.  Case (ii), $j=10$, $l<6$,
	yields 
	$$
	\begin{aligned}
		S(i_*,i)&= N(10,4) + N(6-1,4)- 5 + 4 \chi_{k_2<k_1} + m_1\chi_{k_2=k_1}\\
		&= (36 + 4) + (20+4) - 5 + 4 \chi_{k_2<k_1} + m_1\chi_{k_2=k_1}
		\geq 59 ,
	\end{aligned}
	$$
	giving $\tilde p(i_1^*,i)= \frac{ i_1^*-S(i_1^*, i}{N-M}\geq \frac{669 - 59}{1225}= \frac{610}{1225}<1/2$.
	In case (iii), $j=10$, $l=6$, $k<4$, giving
	$$
	\begin{aligned}
		S(i_*,i)&= N(10,4-1) + N(6,4)- 5 +  \chi_{m_2\leq 1}\\
		&= (4(9)+ 3) + (4(5)+ 4) -5 \chi_{m_2\leq 1}
		\geq  58,
	\end{aligned}
	$$
	giving $\tilde p(i_1^*,i)= \frac{ i_1^*-S(i_1^*, i}{N-M}\geq \frac{669 - 58}{1225}= \frac{611}{1225}<1/2$.
	In case (iv), $m<m_1=1$, and so this case does not occur.
	Combining, we find that $\tilde p(i_1^*,i)<1/2$ for all $i<i_2^*$ with $j=10$, and so also
	$\bar p(i_1^*,i)<1/2$ for all $i<i_2^*$ with $j=10$, verifying the optimality condition for $i<i_1^*$.

	\medskip

	({\it Case 2. $i>i_*$.}) Again, we consider each of cases (i)-(iv) of \eqref{Sform} in turn.
	Case (i), $j>10$ again does not occur within range, and so may be ignored.
	Case (ii), $j=10$, $l>6$, gives
	$$
	\begin{aligned}
		S(i, i_*)&
	N(j,4) + N(l-1,4)- 5 + 4 \chi_{4<k} + m\chi_{4=k}\\
		&= (36+4) + (4(l-2)+ 4) -5 + m\chi_{4=k}\\
		&\leq 
		%40 + 4l- 8 + 4 -5 + 4
	%35+ 4l	
		35 + 24 + 4(l-6)=59 + 4(l-6)
	\end{aligned}
	$$
	But, meanwhile $i-i_*\geq 16(l-6)$, hence
	$$
	\tilde p(i,i_1^*)= \frac{ i_1-S(i, i_1^*}{N-M}\geq \frac{669+ 12(l-6) - 59}{1225}
	\geq \frac{669+ 12 - 59}{1225}>1/2.
	$$
	Finally, case (iv), $j=10$, $l=6$, $k=4$, $m>1$, gives
	$$
	\begin{aligned}
		S(i, i_*)&= N(10,4) + N(6,m)-5\\
		&= (36+4) + (20+ m) -5 
		\leq 55+ m.
	\end{aligned}
	$$
	As $i-i_*\geq m-1$, this gives
	$$
	\tilde p(i,i_1^*)= \frac{ i_1-S(i, i_1^*}{N-M}\geq \frac{669+ m-1 - 55 - m}{1225}
	\geq \frac{669- 56}{1225}
	= \frac{613}{1225} >1/2.
	$$
	Thus, $\tilde p(i,i_1^*)>1/2$ for all $i>i_1^*$ and the optimality condition is satisfied for all
	$i>i_*$ within range.

	\medskip

Therefore, $i_1^*$ guarantees one-shot return $\alpha\geq 0$.
	But it also guarantees $\beta<1$ by comparison with the continuous case.
	Thus, by Theorem \ref{tthm}, $i_1^*$ guarantees a nonnegative return.
\end{proof}

\end{document}